\documentclass[11pt]{amsart}
\usepackage{amsmath}
\usepackage{a4wide}
\usepackage[utf8]{inputenc}
\usepackage{amssymb}
\usepackage{amsopn}
\usepackage{epsfig}
\usepackage{amsfonts}
\usepackage{latexsym}
\usepackage{amsthm}
\usepackage{enumerate}
\usepackage[UKenglish]{babel}
\usepackage{verbatim}
\usepackage{color}
\usepackage{calrsfs}
\usepackage{pgf}
\usepackage{tikz}

\usepackage{mathrsfs}   

 \usetikzlibrary{arrows,automata}

\DeclareMathAlphabet{\pazocal}{OMS}{zplm}{m}{n}
\newtheorem{theorem}{Theorem}[section]
\newtheorem{lemma}[theorem]{Lemma}
\newtheorem{proposition}[theorem]{Proposition}
\newtheorem{corollary}[theorem]{Corollary}
 \newtheorem{main}{Theorem}

\theoremstyle{definition}
\newtheorem{definition}[theorem]{Definition}

\newtheorem{question}[theorem]{Question}

\theoremstyle{remark}
\newtheorem{remark}[theorem]{Remark}

\numberwithin{equation}{section}

\newcommand{\R}{\ensuremath{\mathbb{R}}}
\newcommand{\N}{\ensuremath{\mathbb{N}}}

\newcommand{\A}{\ensuremath{\mathcal{A}}}

\newcommand{\+}{ {\ensuremath{\oplus}}}
\renewcommand{\b}{ {\mathbf{b}}}
\renewcommand{\c}{ {\mathbf{c}}}
\renewcommand{\d}{ {\mathbf{d}}}

\renewcommand{\S}{\ensuremath{\pazocal{S}}}
\renewcommand{\t}{ {\mathbf{t}}}

\newcommand{\re}{\ensuremath{\Theta}}
\renewcommand{\u}{\ensuremath{\pazocal{U}}}

\newcommand{\us}{\mathbf{U}}

\newcommand{\U} {\widetilde{\pazocal{U}}}
\newcommand{\US} {\widetilde{\mathbf{U}}}

\newcommand{\su}{ {\mathbf{u}}}
\newcommand{\sv}{ {\mathbf{v}}}
\newcommand{\sw}{ {\mathbf{w}}}

\newcommand{\Om}{ {\Omega}}

\renewcommand{\O}{ \ensuremath{\pazocal{O}}}

\newcommand{\set}[1]{\left\{#1\right\}}

\newcommand{\la}{\lambda}
\newcommand{\ga}{\gamma}

\newcommand{\f}{\infty}
\newcommand{\de}{\delta}

\newcommand{\lle}{\preccurlyeq}
\newcommand{\lge}{\succcurlyeq}
\renewcommand{\a}{ \mathbf{a}}
\newcommand{\si}{\sigma}

\newcommand{\ra}{\rightarrow}

\begin{document}

\title{Critical base for the unique codings   of fat Sierpinski gasket}

\author{Derong Kong}
\address[D. Kong]{College of Mathematics and Statistics, Chongqing University, 401331, Chongqing, P.R.China}
\email{derongkong@126.com}
 
\author{Wenxia Li}
\address[W. Li]{School of Mathematical Sciences, Shanghai Key Laboratory of PMMP, East China Normal University, Shanghai 200062,
People's Republic of China}
\email{wxli@math.ecnu.edu.cn}

\date{\today}
\dedicatory{}


\subjclass[2010]{37B10, 68R15, 11A63, 28A80}

\begin{abstract}
  Given $\beta\in(1,2)$ the  fat Sierpinski gasket $\pazocal S_\beta$ is  the self-similar set in $\mathbb R^2$ generated by the iterated function system (IFS) 
  \[
  f_{\beta,d}(x)=\frac{x+d}{\beta},\quad d\in\mathcal A:=\set{(0, 0), (1,0), (0,1)}.
  \]
  Then for each point $P\in\pazocal S_\beta$ there exists a sequence $(d_i)\in\mathcal A^\mathbb N$ such that $P=\sum_{i=1}^\f d_i/\beta^i$, and the infinite sequence $(d_i)$ is called a \emph{coding} of $P$. In general, a point in $\pazocal S_\beta$ may have multiple codings since the overlap region $\pazocal O_\beta:=\bigcup_{c,d\in\mathcal A, c\ne d}f_{\beta,c}(\Delta_\beta)\cap f_{\beta,d}(\Delta_\beta)$ has non-empty interior,   where $\Delta_\beta$ is the convex hull of $\pazocal S_\beta$. In this paper we are interested in the invariant set 
  \[
  \widetilde{\pazocal U}_\beta:=\left\{\sum_{i=1}^\infty \frac{d_i}{\beta^i}\in \pazocal S_\beta: \sum_{i=1}^\infty\frac{d_{n+i}}{\beta^i}\notin\pazocal O_\beta~\forall n\ge 0\right\}.
  \]
  Then each point in $  \widetilde{\pazocal U}_\beta$ has a unique coding. We   show that  there is a transcendental number $\beta_c\approx 1.55263$ related to  the Thue-Morse  sequence, such that 
 $\widetilde{\pazocal U}_\beta$ has positive Hausdorff dimension if and only if $\beta>\beta_{c}$.  Furthermore, for $\beta=\beta_c$ the  set $\widetilde{\pazocal U}_\beta$ is uncountable but has zero Hausdorff dimension, and for $\beta<\beta_c$ the  set $\widetilde{\pazocal U}_\beta$ is at most countable. Consequently, we  also answer a conjecture of Sidorov (2007).  Our strategy  is using combinatorics on words based on the lexicographical characterization of $\widetilde{\pazocal U}_\beta$.
  \end{abstract}

\keywords{Fat Sierpinski gasket, unique coding, critical base, Thue-Morse sequence, transcendental number.}
\maketitle

\section{Introduction}\label{s1}
Given $\beta>1$,  let  $\S_\beta$ be the  Sierpinski gasket  in $\R^2$  generated by the \emph{iterated function system} (IFS)
 \[
 f_{\beta,d}(x)=\frac{x+d}{\beta},\quad d\in \A:=\set{(0, 0), (1,0), (0, 1)}.
 \]
 In other words, $\S_\beta$ is the unique non-empty compact set in $\R^2$ satisfying 
 $\S_\beta=\bigcup_{d\in\A} f_{\beta,d}(\S_\beta)$. So for each point $P\in\S_\beta$ there exists a sequence $(d_i)\in\A^\N$ such that  
  \[
 P=\lim_{n\ra\f}f_{\beta,d_1\ldots d_n}({\bf 0})=\sum_{i=1}^\f\frac{d_i}{\beta^i}=:((d_i))_\beta,
 \]
 where $f_{\beta,d_1\ldots d_n}:=f_{\beta, d_1}\circ\cdots \circ f_{\beta, d_n}$ is the composition of $f_{\beta, d_1},\ldots, f_{\beta, d_n}$, and ${\bf 0}=(0, 0)$ is the zero vector in $\R^2$.  
 The infinite sequence $(d_i)\in\A^\N$ is called a \emph{coding} of $P$ with respect to the \emph{alphabet} $\A$. Therefore, 
 the Sierpinski gasket $\S_\beta$ can be  rewritten   as 
 \[
\S_\beta=\set{\sum_{i=1}^\f\frac{d_i}{\beta^i}: d_i\in\A\textrm{ for all }i\ge 1}.
 \]
  
   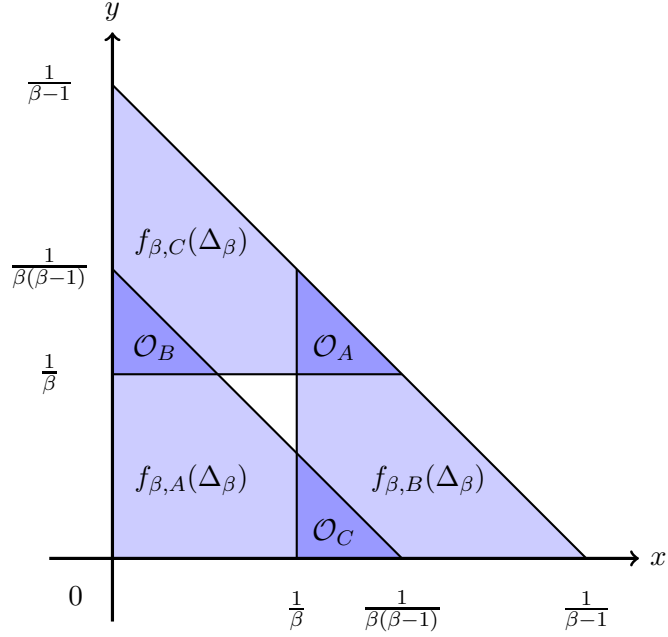
\begin{figure}[h!]
\begin{center}
\begin{tikzpicture}[
    scale=7,
    axis/.style={very thick, ->},
    important line/.style={thick},
    dashed line/.style={dashed, thin},
    pile/.style={thick, ->, >=stealth', shorten <=2pt, shorten
    >=2pt},
    every node/.style={color=black}
    ]


      \node[] at (-0.12,0.9){$\frac{1}{\beta-1}$};
         \node[] at (-0.12,0.55){$\frac{1}{\beta(\beta-1)}$};
         \node[] at (-0.12,0.35){$\frac{1}{\beta}$};
         
       \fill[blue!40](0,0.35)--(0.2,0.35)--(0,0.55)--cycle;
         \fill[blue!40](0.35,0)--(0.55,0)--(0.35,0.2)--cycle;     
      \fill[blue!40](0.35,0.35)--(0.55,0.35)--(0.35,0.55)--cycle;    
      
      \fill[blue!20](0,0)--(0.35,0)--(0.35,0.2)--(0.2,0.35)--(0,0.35)--cycle;   
      
            \fill[blue!20](0.55,0)--(0.9,0)--(0.55,0.35)--(0.35,0.35)--(0.35,0.2)--cycle;  
            
                  \fill[blue!20](0,0.55)--(0.2,0.35)--(0.35,0.35)--(0.35,0.55)--(0,0.9)--cycle;  
             \draw[axis] (-0.12,0)  -- (1.0,0) node(xline)[right]
        {$x$};
    \draw[axis] (0,-0.12) -- (0,1.0) node(yline)[above] {$y$};
     
    \node[] at (-0.07,-0.07){$0$};
             
                    \draw[important line]  (0,0.9)--(0.9,0);
    \draw[important line] (0,0.55)--(0.55,0);
    \draw[important line] (0.35, 0)--(0.35,0.55);
    
       \draw[important line] (0, 0.35)--(0.55,0.35);
       
   \node[] at (0.9, -0.1){$\frac{1}{\beta-1}$};
   
     \node[] at (0.55, -0.1){$\frac{1}{\beta(\beta-1)}$};
      \node[] at (0.35, -0.1){$\frac{1}{\beta}$};

   \node[] at (0.42, 0.05){$\O_C$};
      \node[] at (0.42, 0.4){$\O_A$};
         \node[] at (0.08, 0.4){$\O_B$};
     
        \node[] at (0.15, 0.15){$f_{\beta,A}(\Delta_\beta)$};
          \node[] at (0.15, 0.6){$f_{\beta,C}(\Delta_\beta)$};
              \node[] at (0.6, 0.15){$f_{\beta,B}(\Delta_\beta)$};
\end{tikzpicture} 
\end{center}
\caption{The figure of the first generation of $\S_\beta$ with $\beta=18/11\approx 1.63636$.   The convex hull $\Delta_\beta$ is the triangle with vertices $(0, 0), (1/(\beta-1), 0)$ and $(0, 1/(\beta-1))$. Then $f_{\beta, A}(\Delta_\beta)$ is the left-bottom triangle, $f_{\beta, B}(\Delta_\beta)$ is the right-bottom triangle, and $f_{\beta, C}(\Delta_\beta)$ is the top triangle. The overlap region is  $\O_\beta=\O_A\cup\O_B\cup\O_C$.}\label{Fig:5}
\end{figure}
  
 When $\beta>2$, it is easy to check that the \emph{overlap region}    (see Figure \ref{Fig:5})
\[ 
\O_\beta:=\bigcup_{c,d\in\A, c\ne d}f_{\beta, c}(\Delta_\beta)\cap f_{\beta, d}(\Delta_\beta) 
\]
is empty, where $\Delta_\beta$ is the convex hull of $\S_\beta$. In fact $\Delta_\beta$ is the triangle with vertices $(0, 0), (1/(\beta-1), 0)$ and $(0, 1/(\beta-1))$. So
 for $\beta> 2$ the IFS $\set{f_{\beta, d}(\cdot): d\in\A}$ satisfies the strong separation condition, and then  each point in $\S_\beta$ has a unique coding (cf.~\cite{Falconer_1990}).  When $\beta=2$, one can see that the overlap region $\O_\beta$ consists of three points. Then the IFS $\set{f_{\beta, d}(\cdot): d\in\A}$ fails the strong separation condition, but   still satisfies the open set condition.   In this case, excluding countably many points in $\S_\beta$ having precisely two codings all other points in $\S_\beta$ have a unique coding.  
 
 However, when $\beta\in(1,2)$ the overlap region $\O_\beta$ is non-trivial and it contains interior points. In this case we call $\S_\beta$ a \emph{fat} Sierpinski gasket, and the IFS $\set{f_{\beta, d}(\cdot): d\in\A}$ fails the open set condition (see \cite[Remark 2.3]{Sidorov_2007}). Furthermore, the set of points in $\S_\beta$ with multiple  codings has positive Hausdorff dimension.  In particular, for $\beta\in(1, 3/2]$ the Sierpinski gasket $\S_\beta$ coincides with its convex hull $\Delta_\beta.
$
In this case Lebesgue almost every point in $\S_\beta$ has a continuum of codings (see \cite[Theorem 3.5]{Sidorov_2007}).
When $\beta\in(3/2, 2)$ the structure of  $\S_\beta$ gets  more complicated. Broomhead, Montaldi and Sidorov showed in  \cite{Bro-Mon-Sid-04} that for  $\beta\le \beta_*$ the set $\S_\beta$ has non-empty interior, where $\beta_*\approx 1.54369$ is the appropriate zero of $x^3-2x^2+2x-2$. Some recent progress in this direction can be found in \cite{Hasselblatt-Plante-2014}.  

When $\beta\in(1,2)$ the fat Sierpinski gasket $\S_\beta$ has attracted much more attention in the past twenty years. Simon and Solomyak \cite{Simon-Solomyak-02} showed that there exists a dense set of $\beta\in(1,2)$ such that $\dim_H\S_\beta<\log 3/\log\beta$, where $\log 3/\log \beta$ is the self-similar dimension of $\S_\beta$. We emphasize that $\dim_H\S_\beta=\log 3/\log\beta$ for all $\beta\ge 2$. On the other hand, Jordan \cite{Jordan-05} showed that $\dim_H\S_\beta=\log 3/\log\beta$ for Lebesgue almost every $\beta\in(3/4^{1/3}, 2)\approx(1.88988, 2)$. Furthermore, he and Pollicott \cite{Jordan-Pollicott-06} proved that $\S_\beta$ has positive Lebesgue measure for   almost every $\beta\in(1,1.70853)$.

Let $\beta\in(1,2)$.  In this paper we are interested in an invariant subset of the fat Sierpinski gasket $\S_\beta$ consisting of all points  whose orbits never enter  the overlap region $\O_\beta$. More precisely, we will investigate the \emph{intrinsic univoque set}
\[
\U_\beta:=\set{\sum_{i=1}^\f\frac{d_i}{\beta^i}\in \S_\beta: ~\sum_{i=1}^\f\frac{d_{n+i}}{\beta^i}\,\notin\, \O_\beta\quad\textrm{for all }n\ge 0}.
\]
Observe that if a point $\sum_{i=1}^\f d_i/\beta^i\in \S_\beta$ has multiple codings, then its orbit $\sum_{i=1}^\f d_{n+i}/\beta^i$ must enter the overlap region $\O_\beta$ for some $n\ge 0$. This implies that each point in $\U_\beta$ has a unique coding. Denote the \emph{univoque set} by
\[
\u_\beta:=\set{P\in \S_\beta: P\textrm{ has a unique coding with alphabet } \A}.
\]
Then $\U_\beta\subseteq\u_\beta$ for each $\beta\in(1,2)$. When $\beta\in(1,3/2]$ or $\beta$ is a multinacci number, the equality $\U_\beta=\u_\beta$ holds  (see Proposition \ref{lem:22} below).
 This  explains why  we call $\U_\beta$  the  ``intrinsic univoque set''. 
 
 Let $\beta_G\approx 1.46557$ be the unique root in $(1,2)$ of the equation $x^3-x^2-1=0$. The following result was proved by Sidorov \cite[Theorem 4.1]{Sidorov_2007}.
 \begin{theorem}[Sidorov, 2007]\label{th:sidorov}
 If $\beta\in(1,\beta_G]$, then  
 \[\u_\beta =\set{(0, 0), \left(\frac{1}{\beta-1}, 0\right), \left(0, \frac{1}{\beta-1}\right)}.\] 
 \end{theorem} 
 Furthermore, he  conjectured in \cite[Remark 4.3]{Sidorov_2007} that  the univoque set $\u_\beta$ is at most countable when $\beta\in(\beta_G,3/2]$.  
 Our first result answers his conjecture  affirmatively. 
 \begin{main}\label{th:S=D}
 If $\beta\in(\beta_G, 3/2]$, then 
 \[
 \u_\beta=\set{(0, 0), \left(\frac{1}{\beta-1}, 0\right), \left(0, \frac{1}{\beta-1}\right)}\cup \bigcup_{n=0}^\f\bigcup_{P\in \pazocal C_n, Q\in \pazocal D}\set{P+\frac{Q}{\beta^n(\beta^3-1)}},
 \]
 where 
 \begin{align*}
 \pazocal C_n&:=\set{(0,0), \left(\sum_{i=1}^n\frac{1}{\beta^i}, 0\right), \left(0, \sum_{i=1}^n\frac{1}{\beta^i}\right)}, \\
 \pazocal D&:=\set{(1,\beta), (\beta, 1), (1,\beta^2), (\beta^2, 1), (\beta, \beta^2), (\beta^2, \beta)}.
 \end{align*}
  \end{main}

In order to describe the set  $\U_\beta$ for $\beta\in(3/2, 2)$, we introduce  a     Thue-Morse  type sequence $(\la_i)\in\set{0,1}^\N$. Let $\re$ be the block map defined on the set $\Om:=\set{000, 001,100, 101}$ by 
\[
\re: \Om \ra \Om;\quad 000\mapsto 101, ~001\mapsto 100,~ 100\mapsto 001,~ 101\mapsto 000.
\]
Then for a word $\a=a_1\ldots a_m\in\Om^m$ we set $\re(\a):=\re(a_1)\cdots\re(a_m)$ as the block obtained by concatenating blocks $\re(a_1), \ldots, \re(a_m)$. We emphasize that each digit $a_i$ is a block of length $3$ from $\Om$.  
Let $\t_1=100\in\Om$, and for $n\ge 1$ we set
\[
\t_{n+1}:=\t_n^+\re(\t_n^+),
\]
where $\t_n^+$ denotes the word by changing the last digit  of $\t_n$   from zero to one. 
For example, 
\[\t_2=101000, \quad\t_3=101001\,000100, \quad\t_4=101001000101\; 000100101000,\quad\cdots.\]
 Then the sequence $(\t_n)$ has a componentwise   limit, denoted by $(\la_i)$. So $(\la_i)$   is an infinite sequence in $\set{0, 1}^\N$ related to the classical Thue-Morse sequence  
 (cf.~\cite{Allouche_Shallit_1999}).
One can check that the sequence  $(\la_i)$ begins with $\t_n^+$ for any $n\ge 1$
and $\la_{3k+2}=0$ for all $k\ge 0$.
  
  Let  $\beta_{c}\approx 1.55263$ be the unique root in $(1,2)$ of the equation
\begin{equation}\label{eq:betaG-betaC}
 \sum_{i=1}^\f\frac{\la_i}{x^i}=1.
\end{equation}
In view of Theorem \ref{th:sidorov}, our second result describes the size of   $\U_\beta$ for $\beta\in(\beta_G,2)$. 

\begin{main}
\label{thm:3}\mbox{}
 The number $\beta_c$ is transcendental. 
\begin{enumerate}
[{\rm(i)}]
\item If $\beta\in(\beta_G,\beta_{c})$, then $\U_\beta$ is   countably infinite; 
\item If $\beta=\beta_{c}$, then $\U_{\beta_{c}}$ is uncountable but has zero Hausdorff dimension;
\item If $\beta\in(\beta_{c}, 2)$, then $\U_{\beta}$ has positive Hausdorff dimension. 
\end{enumerate}
\end{main}

\begin{remark}\mbox{}
\begin{itemize}
\item Theorem \ref{thm:3} can be viewed as an analogue of the main results  of Glendinning and Sidorov \cite{Glendinning_Sidorov_2001} for the one dimensional  unique $\beta$-expansions. Then $\beta_c$ is an analogue of the \emph{Komornik-Loreti constant} first investigated by Komornik and Loreti in \cite{Komornik-Loreti-1998}. 

\item  Our proof of Theorem \ref{thm:3} is using word combinatorics  based on the lexicographical characterization of $\U_\beta$ (see Proposition \ref{prop:21}). 
Our method allows us to give a complete description of   $\U_\beta$ for $\beta\in(1, \beta_c]$. 
\end{itemize}
\end{remark}

The rest of the  paper is   organized  as follows. In Section 2 we give the lexicographical characterization of the  intrinsic univoque set $\U_\beta$. Based on this characterization we give an alternate proof of Theorem \ref{th:sidorov} in Section \ref{sec:S=D}. Furthermore, we  prove Theorem \ref{th:S=D} which provides an affirmative answer to a conjecture of Sidorov. In Section \ref{sec:admissible} we investigate all possible admissible words in $\U_{\beta_c}$   based on the three types of Thue-Morse words with alphabet  $\A$. The proof of the main result Theorem \ref{thm:3} is given in Section \ref{s6}. In the final section we pose some questions.

 \section{Characterization of  $\U_\beta$ and a sufficient condition for $\u_\beta=\U_\beta$}\label{s2}

 Given $\beta\in(1,2)$, recall that the Sierpinski gasket $\S_\beta$ is the self-similar set in $\R^2$ generated by the IFS 
 \[
 f_{\beta, d}(x)=\frac{x+d}{\beta},  \quad d\in\A=\set{A, B, C},
 \]
 where 
 $
 A:=(0,0),  B:=(1,0)$ and  $C:=(0,1)$ are the  digits.     Then for any point $P\in \S_\beta$ there exists a sequence $(d_i)\in\A^\N$ such that 
 \begin{equation}\label{eq:coding}
 P=((d_i))_\beta=\sum_{i=1}^\f\frac{d_i}{\beta^i}.
 \end{equation}
 Note that each point $P$ in $\R^2$ can be written as $P=(P^1, P^2)$, where   $P^1$ and $P^2$ are the first and  second coordinates of $P$. Similarly, each digit $d_i\in\A$ can be written as $d_i=(d_i^1, d_i^2)$. Then (\ref{eq:coding}) can be rewritten coordinately as  
 \[
 P^1=((d_i^1))_\beta:=\sum_{i=1}^\f\frac{d_i^1}{\beta^i} \quad\textrm{and}\quad P^2=((d_i^2))_\beta:= \sum_{i=1}^\f\frac{d_i^2}{\beta^i}.
 \]
   
In view of Figure \ref{Fig:5},  the convex hull $\Delta_\beta$ of $\S_\beta$   is the triangle  with three vertices $(0, 0), (1/(\beta-1), 0)$ and $(0, 1/(\beta-1))$.
 Then $f_{\beta, A}(\Delta_\beta)$ is the left-bottom triangle, $f_{\beta, B}(\Delta_\beta)$ is the right-bottom triangle, and $f_{\beta, C}(\Delta_\beta)$ is the top triangle.  Since $1<\beta<2$, the overlap region $\O_\beta=\O_A\cup\O_B\cup\O_C$ has non-empty interior, where 
 \begin{align*}
 \O_A:=f_{\beta, B}(\Delta_\beta)\cap f_{\beta, C}(\Delta_\beta),\quad\O_B=f_{\beta, A}(\Delta_\beta)\cap f_{\beta, C}(\Delta_\beta),\quad\O_C=f_{\beta, A}(\Delta_\beta)\cap f_{\beta, B}(\Delta_\beta).
 \end{align*}
 Indeed, each of the sets $\O_A, \O_B$ and $\O_C$ has non-empty interior for   $\beta\in(1,2)$. 

 In order to  describe the intrinsic univoque set $\U_\beta$, we   need some notation from symbolic dynamics. For a word $\c=c_1\ldots c_n\in\set{0,1}^*$ we mean a finite string of zeros and ones. For an integer $k\ge 1$ we denote by $\c^k:=\c\c\cdots\c$ the $k$-times concatenation of $\c$ with itself, and we write for $\c^\f$ the periodic sequence with periodic block $\c$. For a word $\c=c_1\ldots c_n$ with $c_n=0$ we denote by $\c^+:=c_1\ldots c_{n-1}(c_n+1)$.
 For a sequence $(c_i)\in \set{0,1}^\N$ we define its \emph{reflection} by
 $
 \overline{(c_i)}:=(1-c_1)(1-c_2)\cdots.
 $ 
 Throughout the paper we will use the lexicographical order between sequences and words. More precisely,  for two sequences $(c_i), (d_i)\in\set{0,1}^\N$ we write $(c_i)\prec (d_i)$ or $(d_i)\succ (c_i)$ if $c_1<d_1$, or there exists $k\ge 2$ such that $c_i=d_i$ for all $1\le i<k$ and $c_k<d_k$. Similarly, we write $(c_i)\lle (d_i)$ or $(d_i)\lge(c_i)$ if $(c_i)\prec (d_i)$ or $(c_i)=(d_i)$.  Furthermore, for two words $\c, \d\in\set{0, 1}^*$ we say $\c\prec \d$ if $\c 0^\f\prec \d 0^\f$. 

Given $\beta\in(1,2]$ let $\de(\beta)=(\de_i(\beta))\in\set{0,1}^\N$ be the \emph{quasi-greedy} $\beta$-expansion of $1$ (cf.~\cite{Daroczy_Katai_1993}), i.e., the lexicographically largest $\beta$-expansion of $1$ not ending with a string of zeros. The following characterization of  $\de(\beta)$ was essentially due to Parry \cite{Parry_1960} (see also \cite{Allaart-Baker-Kong-17, Baiocchi_Komornik_2007}).
\begin{lemma}\label{lem:delta-beta}
\begin{enumerate}
[{\rm(i)}]
\item The map $\beta\mapsto \de(\beta)$ is a strictly increasing bijection from $(1,2]$ onto the set $\mathbf{A}$ of all sequences $(a_i)\in\set{0,1}^\N$ not ending with $0^\f$ and satisfying
\[
a_{n+1}a_{n+2}\ldots \lle a_1 a_2\ldots\quad\textrm{for all }n\ge 0.
\]

\item The inverse map 
\[
\de^{-1}: \mathbf A\ra (1,2];\quad (a_i)\mapsto \de^{-1}((a_i))
\]
is bijective and strictly increasing. Furthermore, $\de^{-1}$ is continuous. 
\end{enumerate}
\end{lemma}

Recall that $\U_\beta$ is the intrinsic univoque set consisting of all points  $((d_i))_\beta\in\S_\beta$  satisfying 
$
((d_{n+i}))_\beta\notin\O_\beta$ {for all} $n\ge 0.
$
  Let 
\[
\US_\beta:=\set{(d_i)\in\A^\N: ((d_i))_\beta\in\U_\beta}
\]
be the set of all codings of points from $\U_\beta$. Note that  $\U_\beta\subseteq \u_\beta$ for all $\beta\in(1,2)$. This means each point in $\U_\beta$ has a unique coding.
So the map $(d_i)\mapsto ((d_i))_\beta$ is  bijective  from $\US_\beta$ to $\U_\beta$.

For a digit $d=(d^1, d^2)\in\A$ we write $d^\+:=d^1+d^2$. Then  $d^1, d^2, d^\+\in\set{0, 1}$.  In the following  we give a lexicographical characterization of $\US_\beta$, or equivalently, of $\U_\beta$.

 \begin{proposition}
 \label{prop:21}
Let $\beta\in(1,2)$. Then     $(d_i)\in\US_\beta$ if and only if   $(d_i)\in\A^\N$ satisfies 
\[\left\{\begin{array}{lll}
d_{n+1}^1d_{n+2}^1\ldots\prec \de(\beta)&\quad\textrm{whenever}\quad& d_n^1=0,\\
d_{n+1}^2 d_{n+2}^2\ldots\prec \de(\beta)&\quad\textrm{whenever}\quad& d_n^2=0,\\
d_{n+1}^\+d_{n+2}^\+\cdots \succ\overline{\de(\beta)}&\quad\textrm{whenever}\quad& d_n^\+=1.
\end{array}\right.\]
 \end{proposition}
 Before proving the proposition we point out that each coding $(d_i)\in\US_\beta$ is an infinite sequence of vectors from $\A$, while the characterization is based on   the projections of $(d_i)$ which are infinite sequences of zeros and ones.
 \begin{proof}
 
First we prove the necessity. Take a sequence $(d_i)\in\US_\beta$. We consider three cases.

Case I. Suppose $d_n^1=0$ for some $n\ge 1$. Then $d_n=A$ or $d_n=C$.  In view of Figure \ref{Fig:5},  it follows    that 
\[(d_nd_{n+1}\ldots)_\beta\in (f_{\beta, A}(\Delta_\beta)\cup f_{\beta, C}(\Delta_\beta))\setminus\O_\beta,\]
 which implies  
 $
  (d_n^1d_{n+1}^1\ldots)_\beta<{1}/{\beta}.
 $
Whence,
 \begin{equation}\label{eq:prop1}
 (d_{n+1}^1d_{n+2}^1\ldots)_\beta<1=(\de(\beta))_\beta.
 \end{equation}
 If $\de(\beta)$ is not periodic, then $\de(\beta)$ is also the greedy $\beta$-expansion of $1$, which is the lexicographically largest  $\beta$-expansion of $1$ (cf.~\cite{Daroczy_Katai_1993}). So (\ref{eq:prop1}) implies
 \begin{equation}\label{eq:prop2}
 d_{n+1}^1d_{n+2}^1\ldots\prec \de(\beta)
 \end{equation}
 as required.
 
If $\de(\beta)$ is periodic, say $\de(\beta)=(\de_1\ldots \de_m)^\f$ with $m$ the smallest period, then $\de_m=0$, and the greedy $\beta$-expansion of $1$ is given by $\de_1\ldots \de_m^+ 0^\f$. By (\ref{eq:prop1}) it follows that 
$
d_{n+1}^1d_{n+2}^1\ldots\prec \de_1\ldots\de_m^+ 0^\f,
$
which yields 
\[d_{n+1}^1\ldots d_{n+m}^1\lle \de_1\ldots\de_m.\]
 If the strict inequality holds, then (\ref{eq:prop2}) holds and we are done. Suppose $d_{n+1}^1\ldots d_{n+m}^1=\de_1\ldots \de_m$. Then $d_{n+m}^1=0$. By the same argument as above, we either have (\ref{eq:prop2}) or $d_{n+1}^1\ldots d_{n+2m}^1=(\de_1\ldots \de_m)^2$. 
 
 Repeating this procedure indefinitely we either have (\ref{eq:prop2}) or  $d^1_{n+1}d_{n+2}^1\ldots=(\de_1\ldots \de_m)^\f=\de(\beta)$. Note  by (\ref{eq:prop1})   the second case can not occur. This proves (\ref{eq:prop2}) if $d_n^1=0$. 
 
Case II. Suppose  $d_n^2=0$ for some $n\ge 1$. Then by a similar argument as in Case I  we can prove $d_{n+1}^2d_{n+2}^2\ldots\prec \de(\beta)$. 
  
Case III. Suppose $d_n^\+=1$, i.e., $d_n^1=1$ or $d_n^2=1$. By symmetry we may assume $d_n^1=1$. Then $d_n=(1,0)=B$. In view of Figure \ref{Fig:5},   it follows   that  
 $(d_nd_{n+1}\ldots)_\beta\in f_{\beta, B}(\Delta_\beta)\setminus\O_\beta$, which implies 
 \[
(1d_{n+1}^\+d_{n+2}^\+\ldots)_\beta= (d_n^1d_{n+1}^1\ldots)_\beta+(d_n^2d_{n+1}^2\ldots)_\beta>\frac{1}{\beta(\beta-1)}.
 \]
 This is equivalent to 
 \[
 \frac{1}{\beta-1}-(d_{n+1}^\+ d_{n+2}^\+\ldots)_\beta<1.
 \]
 Whence,
 \[
 (\overline{d_{n+1}^\+ d_{n+2}^\+\ldots})_\beta<1.
 \]
 Observe that $\overline{d_i^\+}=1-d_i^\+\in\set{0,1}$ for any $i\ge 1$, and $\overline{d_n^\+}=0$. Then by the same argument as in the case $d_n=0$ we can show that 
 \[
\overline{d_{n+1}^\+d_{n+2}^\+\ldots}\prec \de(\beta).
 \]
 In other words,
 $
d_{n+1}^\+d_{n+2}^\+\ldots \succ\overline{\de(\beta)}.
 $  This proves the necessity.

Now we turn to prove the sufficiency. Let $(d_i)\in\A^\N$ be the sequence satisfying the inequalities in the proposition. Then  it suffices to prove 
\begin{equation}\label{eq:kr1}
(d_nd_{n+1}\ldots)_\beta\in f_{\beta, d_n}(\Delta_\beta)\setminus \O_\beta\quad\textrm{for all }n\ge 1.
\end{equation}
Fix $n\ge 1$. Suppose $d_n=A=(0,0)$. Then   for $\ell\in\set{1,2}$ we have 
 \begin{equation}\label{eq:k1}
 d_{m+1}^\ell d_{m+2}^\ell\ldots\prec\de(\beta)\quad\textrm{whenever}\quad d_m^\ell=0.
 \end{equation}
 Since $d_n^\ell=0$, we claim that $ (d_n^\ell d_{n+1}^\ell\ldots)_\beta<{1}/{\beta}$.
 
 Starting with $k_0:=n$ we define by recurrence a sequence of indices $k_0<k_1<\cdots$ satisfying for $j=1,2,\ldots$ the conditions
 \[
 d_{k_{j-1}+1}^\ell\ldots d_{k_j-1}^\ell=\de_1(\beta)\ldots \de_{k_j-k_{j-1}-1}(\beta)\quad\textrm{and}\quad d_{k_j}^\ell<\de_{k_j-k_{j-1}}(\beta).
 \]
 By \eqref{eq:k1} it follows that we have infinitely many indices $(k_j)$. Then 
 \begin{align*}
 (d_n^\ell d_{n+1}^\ell\ldots)_\beta=\beta^{n-1}\sum_{i=n+1}^\f\frac{d_i^\ell}{\beta^i}&=\beta^{n-1}\sum_{j=1}^\f\sum_{i=1}^{k_j-k_{j-1}}\frac{d_{k_{j-1}+i}^\ell}{\beta^{k_{j-1}+i}}\\
 &=\beta^{n-1}\sum_{j=1}^\f\left(\sum_{i=1}^{k_j-k_{j-1}}\frac{\de_{i}(\beta)}{\beta^{k_{j-1}+i}}-\frac{1}{\beta^{k_j}}\right)\\
 &<\beta^{n-1}\sum_{j=1}^\f\left(\frac{1}{\beta^{k_{j-1}}}-\frac{1}{\beta^{k_j}}\right)\\
 &=\frac{\beta^{n-1}}{\beta^{k_0}}=\frac{1}{\beta}.
 \end{align*}
 So, for $d_n=A$ we have 
 \[(d_n^1 d_{n+1}^1\ldots)_\beta<\frac{1}{\beta}\quad \textrm{and}\quad (d_n^2 d_{n+1}^2\ldots)_\beta<\frac{1}{\beta}.\] In view of Figure \ref{Fig:5} this implies that $(d_nd_{n+1}\ldots)_\beta\in f_{\beta, A}(\Delta_\beta)\setminus\O_\beta$. Hence, (\ref{eq:kr1}) holds if  $d_n=A$.

Since the proof of (\ref{eq:kr1}) for $d_n=C$ is similar to  that for $d_n=B$, it suffices to prove (\ref{eq:kr1}) for $d_n=B=(1,0)$. Note that  
 \[
d_{m+1}^\+d_{m+2}^\+\ldots \succ\overline{\de(\beta)}\quad\textrm{whenever }  d_m^\+=1.
 \]
This is equivalent to
\[
\overline{d_{m+1}^\+d_{m+2}^\+\ldots}\prec \de(\beta)\quad\textrm{whenever } \overline{d_m^\+}=0.
\]
Observe that $\overline{d_n^\+}=1-d_n^\+=0$. By the same argument as in the case for $d_n^\ell=0$ it follows that 
\[
\frac{1}{\beta-1}-(d_n^\+d_{n+1}^\+\ldots)_\beta=(\overline{d_{n}^\+d_{n+1}^\+\ldots})_\beta<\frac{1}{\beta},
\]
which implies
 \begin{equation}\label{eq:kr2}
 (d_n^1d_{n+1}^1\ldots)_\beta+(d_n^2d_{n+2}^2\ldots)_\beta=(d_n^\+d_{n+1}^\+\ldots)_\beta>\frac{1}{\beta(\beta-1)}.
  \end{equation}
  Since $d_n^2=0$, by the above proof it follows that 
  \begin{equation}\label{eq:kr3}
  (d_n^2d_{n+1}^2\ldots)_\beta<\frac{1}{\beta}.
  \end{equation}
In view of Figure \ref{Fig:5}, we obtain  by (\ref{eq:kr2}) and (\ref{eq:kr3}) that 
 $
 (d_nd_{n+1}\ldots)_\beta\in f_{\beta, B}(\Delta_\beta)\setminus\O_\beta.
 $ This proves  (\ref{eq:kr1}) if $d_n=B$, and then completes the proof.   
 \end{proof}
 
 \begin{remark}\label{re:22}\mbox{}
 
 \begin{itemize}
\item Note by Lemma \ref{lem:delta-beta} that the map $\beta\mapsto \de(\beta)$ is strictly increasing in $(1,2)$. Then Proposition \ref{prop:21} implies that the set-valued map $\beta\mapsto\US_\beta$ is increasing, i.e., for $1<\beta_1<\beta_2<2$ we have $\US_{\beta_1}\subseteq\US_{\beta_2}$. Moreover, the symbolic   set $\US_\beta$ is shift invariant, i.e., for any sequence $(d_i)\in\US_\beta$ we have $\si((d_i)):=(d_{i+1})\in\US_\beta$.

\item Observe that  each point $P=(P^1, P^2)\in\U_\beta$ has a unique coding  with respect to the alphabet $\A$. However, this does not mean its projections  $P^1$ and $P^2$  have a unique $\beta$-expansion with respect to the digits set $\set{0, 1}$.   For example, take $\beta\in(1,2)$ such that  $\de(\beta)=(101000)^\f$. Then by Proposition \ref{prop:21} one can check that  the sequence
\[
(d_i)=(BAC)^\f\in\US_\beta.
\]  
However, neither of its coordinate sequences  $(d_i^1)=(100)^\infty$ and $(d_i^2)=(001)^\f$ is a  univoque sequence in the one dimensional sense (cf.~\cite{DeVries_Komornik_2008}).  
\end{itemize}
 \end{remark}
 
At the end of this section we present a sufficient condition for which the intrinsic univoque set $\U_\beta$ coincides with the univoque set $\u_\beta$. A number $\beta\in(1,2)$ is called a \emph{multinacci number} if $\de(\beta)=(1^m0)^\f$ for some positive integer $m$. So by Lemma \ref{lem:delta-beta} the smallest multinacci number is the Golden Ratio ${(1+\sqrt{5})}/{2}\approx 1.61803$ with $\de({(1+\sqrt{5})}/{2})=(10)^\f$.

 \begin{proposition}\label{lem:22}
 If $\beta\in(1,3/2]$ or $\beta$ is a multinacci number, then $\U_\beta=\u_\beta$.
 \end{proposition}
 \begin{proof}
 Observe that $\U_\beta\subseteq\u_\beta$ for all $\beta\in(1,2)$. Then it suffices to prove $\u_\beta\subseteq\U_\beta$ for $\beta\in(1,3/2]$ and for $\beta$ being a multinacci number.
If $\beta\in(1,3/2]$, then   $\S_\beta=\Delta_\beta$. In view of Figure \ref{Fig:5} it follows that 
 \begin{equation}\label{eq:overlap} 
 \begin{split}
 \O_A\cap\S_\beta&=f_{\beta, B}(\S_\beta)\cap f_{\beta,C}(\S_\beta),\\
 \O_B\cap\S_\beta&=f_{\beta, A}(\S_\beta)\cap f_{\beta,C}(\S_\beta),\\
  \O_C\cap\S_\beta&=f_{\beta, A}(\S_\beta)\cap f_{\beta,B}(\S_\beta).\\
 \end{split}
 \end{equation}
Therefore, 
  any point in $\O_\beta\cap\S_\beta=(\O_A\cup\O_B\cup\O_C)\cap\S_\beta$ has at least two codings. For example, any point in $\O_A\cap \S_\beta$  has at least two codings: one begins with digit $B$ and the other begins with digit $C$. Hence, $\u_\beta\subseteq\U_\beta$ if $\beta\in(1,3/2]$.
  
  If $\beta$ is a multinacci number,   then $\S_\beta$ is a proper subset of $\Delta_\beta$. But we still have (\ref{eq:overlap}) since by \cite[Theorem 3.3]{Bro-Mon-Sid-04}  
  \[
  f_{\beta, d}(\Delta_\beta)\cap \S_\beta=f_{\beta, d}(\S_\beta)\quad\textrm{for any }d\in\A.
  \]
  Again by the same argument as above we conclude that any point in $\O_\beta\cap\S_\beta$ has at least two codings. This proves $\u_\beta\subseteq\U_\beta$ for $\beta$ being a multinacci number, and completes the proof.
   \end{proof}

 \section{Description of $\U_\beta$ when $\S_\beta=\Delta_\beta$}\label{sec:S=D}
 In this section we will investigate the intrinsic univoque set $\u_\beta$ when $\S_\beta=\Delta_\beta$. As a result we give an alternate proof of Theorem \ref{th:sidorov}. In the second part of this section we prove  Theorem \ref{th:S=D} which answers a conjecture of Sidorov \cite{Sidorov_2007} affirmatively.
 
   Note that $\S_\beta=\Delta_\beta$ if and only if $\beta\in(1, 3/2]$. Furthermore, by Proposition \ref{lem:22} we have $\u_\beta=\U_\beta$ for any $\beta\in(1, 3/2]$. Hence, it suffices to describe the symbolic set $\US_\beta$ for $\beta\in(1, 3/2]$. Recall  that     $\beta_G\approx 1.46557$ is  the unique root in $(1, 2)$  of $x^3-x^2-1=0$. Then  
\[\de(\beta_G)=(100)^\f.\]
We will describe $\US_\beta$ for   $\beta\le \beta_G$ and $\beta\in(\beta_G, 3/2]$ separately. 
 
  \subsection{Description of $\US_\beta$ for $\beta\le\beta_G$}\label{s3}
 
  In this part we will investigate the intrinsic univoque set $\U_\beta$ for $\beta\le \beta_G$, and then give an alternate proof of Theorem \ref{th:sidorov}. We emphasize that   Sidorov proved this result in \cite{Sidorov_2007} in a dynamical way, while our proof is using combinatorics on words based on the lexicographical characterization of $\US_\beta$ described  in Proposition \ref{prop:21}.  
First we show that the symbolic set $\US_{\beta_G}$ does not contain any sequence with three consecutive distinct elements. 
 
 \begin{lemma}
 \label{lem:31}
 Let $(d_i)\in\US_{\beta_G}$. Then  any   three consecutive elements $d_n, d_{n+1}$ and $d_{n+2}$ can not be all  distinct. 
 \end{lemma}
\begin{proof}
Suppose there exists $n\ge 1$ such that $d_n, d_{n+1}$ and  $d_{n+2}$ are  distinct. Then 
\begin{equation}\label{eq:31}
d_nd_{n+1}d_{n+2}\in\set{BAC, CBA, ACB; ABC, CAB, BCA}.
\end{equation}
  Since the proofs for different cases are similar, we may assume   $d_nd_{n+1}d_{n+2}=BAC$. 

For convenience we represent the block $d_nd_{n+1}d_{n+2}=BAC$ as three points located from the left to the right as in Figure \ref{Fig:6}. Then any point in the first row has its first coordinate equaling $1$, and any point in the third row has its second coordinate equaling $1$. While any point in the middle row has both   coordinates equaling $0$. 
   \begin{figure}[h!]
\begin{center}
\begin{tikzpicture}[
    scale=6,
    axis/.style={very thick, ->},
    important line/.style={thick},
    dashed line/.style={dashed, thin},
    pile/.style={thick, ->, >=stealth', shorten <=2pt, shorten
    >=2pt},
    every node/.style={color=black}
    ]
    \draw[axis] (0,0)  -- (1.5,0) node(xline)[right]
        {};
        \node[] at (-0.2,0){$A=(0,0)$};
        
      \node[] at (-0.2, 0.2){$B=(1,0)$};
         \node[] at (-0.2,-0.2){$C=(0,1)$};
     
 \draw [fill] (0,0.2) circle [radius=0.01];
 \node[]at(0,0.15){$d_n$};
  \draw [fill] (0.2,0) circle [radius=0.01];
   \node[]at(0.2,-0.05){$d_{n+1}$};
   \draw [fill] (0.4,-0.2) circle [radius=0.01];
    \node[]at(0.4,-0.25){$d_{n+2}$};
      
         \draw [fill] (0.6,0.2) circle [radius=0.01];
    \node[]at(0.6,0.15){$d_{n+3}$};
      \draw [fill] (0.8,0) circle [radius=0.01];
   \node[]at(0.8,-0.05){$d_{n+4}$};
   \draw [fill] (1.0,-0.2) circle [radius=0.01];
    \node[]at(1.0,-0.25){$d_{n+5}$};
      \end{tikzpicture} 
\end{center}
\caption{The presentation of the block $d_n\ldots d_{n+5}=BACBAC$.}\label{Fig:6}
\end{figure}
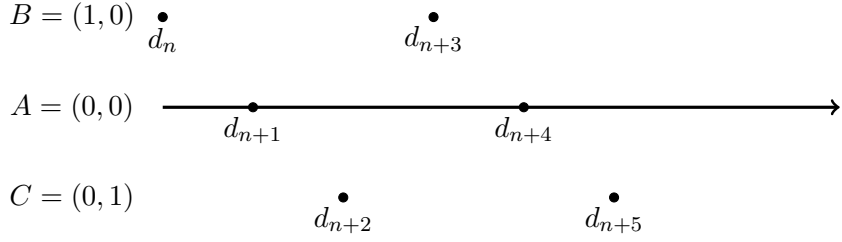

First we claim that $d_{n+3}=B$. Note that $d_nd_{n+1}d_{n+2}=BAC$. If $d_{n+3}=A$, then  $d_{n}^\+\ldots d_{n+3}^\+=1010$. This gives 
\[d_n^\+=1\quad\textrm{and}\quad d_{n+1}^\+d_{n+2}^\+\ldots\prec (011)^\f=\overline{\de(\beta_G)}.\]
 By Proposition \ref{prop:21} it follows that $(d_i)\notin\US_{\beta_G}$. If $d_{n+3}=C$, then $d_{n+1}^2d_{n+2}^2d_{n+3}^2=011$ which yields $d_{n+1}^2=0$ and $d_{n+2}^2d_{n+3}^2\ldots\succ(100)^\f=\de(\beta)$. Again by Proposition \ref{prop:21} we have $(d_i)\notin\US_{\beta_G}$. Therefore, $d_{n+3}=B$. 
 
 Next we claim $d_{n+4}=A$. Note that $d_{n+1}d_{n+2} d_{n+3}=ACB$. Then $d_{n+2}^1d_{n+3}^1d_{n+4}^1=011$ if $d_{n+4}=B$, and  $d_{n+1}^2\ldots d_{n+4}^2=0101$ if $d_{n+4}=C$. In both cases we can deduce by Proposition \ref{prop:21} that $(d_i)\notin\US_{\beta_G}$. So, $d_{n+4}=A$.
 
 Now we claim $d_{n+5}=C$.  Observe that $d_{n+2}d_{n+3}d_{n+4}=CBA$. If $d_{n+5}=B$, then $d_{n+2}^1\ldots d_{n+5}^1=0101$. If $d_{n+5}=A$, then $d_{n+3}^\+d_{n+4}^\+d_{n+5}^\+=100$. In both cases we can infer from Proposition \ref{prop:21} that $(d_i)\notin\US_{\beta_G}$. Hence, $d_{n+5}=C$. 
 
 Observe that in the above arguments each three consecutive  block $d_{m+1}d_{m+2}d_{m+3}$ uniquely determines the next digit $d_{m+4}$.
Repeating this procedure indefinitely one can show that   \[d_nd_{n+1}\ldots=(BAC)^\f.\]
 Then $d_{n+1}^2=0$ and $d_{n+2}^2d_{n+3}^2\ldots=(100)^\f=\de(\beta_G)$. However, by Proposition \ref{prop:21} this implies that $(d_i)\notin\US_{\beta_G}$, leading to contradiction. Hence, any three consecutive elements of $(d_i)\in\US_{\beta_G}$ can not be all distinct. 
\end{proof}
 \begin{proof}
 [An alternate proof of Theorem \ref{th:sidorov}]
 By Proposition \ref{prop:21} it follows that for any $\beta\in(1, \beta_G]$ the sequences $A^\f, B^\f$ and $C^\f$ all belong to $\US_\beta$. Observe  by Proposition \ref{lem:22} that $\u_\beta=\U_\beta$ for $\beta\in(1, \beta_G]$, and by Lemma \ref{lem:delta-beta} and Proposition \ref{prop:21} that the set-valued map $\beta\mapsto\US_{\beta}$ is increasing on $(1, \beta_G]$. So it suffices to show that $\US_{\beta_G}$ consists of the three sequences $A^\f, B^\f$ and $C^\f$. 
 
 Suppose on the contrary that there exists $(d_i)\in\US_{\beta_G}\setminus\set{A^\f, B^\f, C^\f}$. Then there is an integer  $n\ge 1$ such that $d_{n+1}\ne d_n$. So,
 \[
 d_nd_{n+1}\in\set{AB, AC, BA, BC, CA, CB}.
 \]
Since the proofs for different cases are similar, without loss of generality we may assume $d_nd_{n+1}=AB$.  By Lemma \ref{lem:31} the next element $d_{n+2}$ must be equal to $A$ or $B$. We will finish the proof by showing that  $(d_i)\notin\US_{\beta_G}$.  

Case I. $d_{n+2}=A$. By Lemma \ref{lem:31} we have   $d_{n+3}\in\set{A, B}$. Then  
\[
d_n\ldots d_{n+3}=ABAA\quad\textrm{or} \quad d_n\ldots d_{n+3}=ABAB.
\]
In the first case we have $d_{n+1}^\+d_{n+2}^\+d_{n+3}^\+=100$, and in the second case we have $d_n^1\ldots d_{n+3}^1=0101$. In both cases, we can deduce by Proposition \ref{prop:21}   that $(d_i)\notin\US_{\beta_G}$. 

Case II. $d_{n+2}=B$. Then $d_nd_{n+1}d_{n+2}=ABB$, which gives $d_n^1 d_{n+1}^1 d^1_{n+2}=011$. By Proposition \ref{prop:21} this again yields $(d_i)\notin\US_{\beta_G}$. 
 \end{proof}

 \subsection{Description of $\US_\beta$ for $\beta_G<\beta\le 3/2$}\label{s4}
Let $\beta_2\approx1.5385\in(\beta_G, 2)$ such that 
 \[
 \de(\beta_2)= (101000)^\f.
 \]
 Then  $3/2\in(\beta_G, \beta_2)$. By  Theorem \ref{th:sidorov} it suffices to describe the difference set $\US_{3/2}\setminus\US_{\beta_G}$. 
 
 \begin{proposition}
 \label{prop:41}
Let $\beta\in(\beta_G,\beta_2]$. Then any sequence in  $\US_{\beta}\setminus\US_{\beta_G}$ must end with    $(ABC)^\f$ or $(CBA)^\f$. 
 \end{proposition}

 First we show that for $\beta\le \beta_2$ any   block of the form `$cdd$' can not occur in sequences of $\US_\beta$. In the folowing we prove this for $\beta$ in a larger range.
 
 \begin{lemma}\label{l42}
Let $\beta\in(1, 2)$.  If $\de(\beta)$ begins with $10$, then any block from the following set
\[
\mathcal F=\set{BAA, BCC, ABB, ACC, CBB, CAA}
\]
is forbidden in $\US_\beta$.
 \end{lemma}
 \begin{proof}
 Since the proofs for different cases  are similar, without loss of generality we only prove that   $BAA$ is forbidden in $\US_\beta$.
 Suppose on the contrary there exists a sequence $(d_i)\in\US_\beta$ such that $d_nd_{n+1}d_{n+2}=BAA$ for some $n\ge 1$.  Then $d_n^\+d_{n+1}^\+d_{n+2}^\+ =100$. This implies $d_n^\+=1$ and $d_{n+1}^\+d_{n+2}^\+\ldots\prec\overline{\de(\beta)}.$
 By Proposition \ref{prop:21}  we have $(d_i)\notin\US_\beta$, leading to a contradiction. So $BAA$ is forbidden in  $\US_\beta$.
 \end{proof}
 
 In order to prove Proposition \ref{prop:41} we  also need the following lemma.
 \begin{lemma}
 \label{lem:43}
 Let $\beta\in(1,2)$ such that  $\de(\beta)$ begins with $101000$. Suppose $(d_i)\in\US_{\beta}$.
 \begin{enumerate}[{\rm(i)}]
 \item If   $d_{m+1}d_{m+2}d_{m+3}=BAB$ and $d_m\ne d_{m+1}$, then the next block $d_{m+4}d_{m+5}d_{m+6}=CAC$. 
 \item If  $d_{m+1}d_{m+2}d_{m+3}=CAC$ and $d_m\ne d_{m+1}$, then the next block $d_{m+4}d_{m+5}d_{m+6}=BAB$. 
\item If   $d_{m+1}d_{m+2}d_{m+3}=ABA$ and $d_m\ne d_{m+1}$, then the next block $d_{m+4}d_{m+5}d_{m+6}=CBC$. 
\item If  $d_{m+1}d_{m+2}d_{m+3}=CBC$ and $d_m\ne d_{m+1}$, then the next block $d_{m+4}d_{m+5}d_{m+6}=ABA$. 
\item If   $d_{m+1}d_{m+2}d_{m+3}=ACA$ and $d_m\ne d_{m+1}$, then the next block $d_{m+4}d_{m+5}d_{m+6}=BCB$. 
\item If   $d_{m+1}d_{m+2}d_{m+3}=BCB$ and $d_m\ne d_{m+1}$, then the next block $d_{m+4}d_{m+5}d_{m+6}=ACA$. 
 \end{enumerate}
 \end{lemma}
    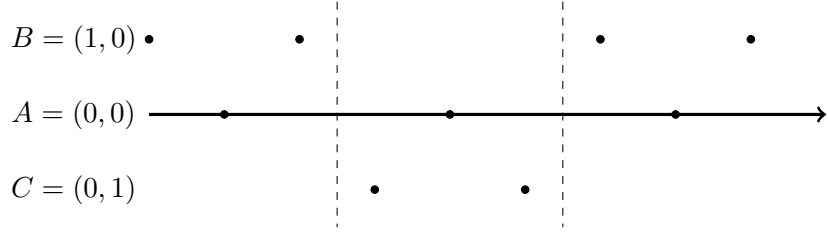
\begin{figure}[h!]
\begin{center}
\begin{tikzpicture}[
    scale=5,
    axis/.style={very thick, ->},
    important line/.style={thick},
    dashed line/.style={dashed, thin},
    pile/.style={thick, ->, >=stealth', shorten <=2pt, shorten
    >=2pt},
    every node/.style={color=black}
    ]
    \draw[axis] (0,0)  -- (1.8,0) node(xline)[right]
        {};
        \node[] at (-0.2,0){$A=(0,0)$};
        
      \node[] at (-0.2, 0.2){$B=(1,0)$};
         \node[] at (-0.2,-0.2){$C=(0,1)$};
     
 \draw [fill] (0,0.2) circle [radius=0.01];
 
  \draw [fill] (0.2,0) circle [radius=0.01];
 
   \draw [fill] (0.4,0.2) circle [radius=0.01];
 
      \draw[dashed line] (0.5, 0.3)-- (0.5,-0.3); 
  
         \draw [fill] (0.6,-0.2) circle [radius=0.01];
 
      \draw [fill] (0.8,0) circle [radius=0.01];
 
   \draw [fill] (1.0,-0.2) circle [radius=0.01];
    \draw[dashed line] (1.1, 0.3)-- (1.1,-0.3);
   
      \draw [fill] (1.2,0.2) circle [radius=0.01];
 
      \draw [fill] (1.4,0) circle [radius=0.01];
 
   \draw [fill] (1.6,0.2) circle [radius=0.01];
 
      \end{tikzpicture} 
\end{center}
\caption{Type-$A$ presentation: $A\mapsto A, B\mapsto C, C\mapsto B$. The first and second columns (i) $BAB \mapsto CAC$; the second and third columns (ii) $CAC \mapsto BAB$.}\label{figure:4}
\end{figure}
    \begin{figure}[h!]
\begin{center}
\begin{tikzpicture}[
    scale=5,
    axis/.style={very thick, ->},
    important line/.style={thick},
    dashed line/.style={dashed, thin},
    pile/.style={thick, ->, >=stealth', shorten <=2pt, shorten
    >=2pt},
    every node/.style={color=black}
    ]
    \draw[axis] (0,0.2)  -- (1.8,0.2) node(xline)[right]
        {};
        \node[] at (-0.2,0){$A=(0,0)$};
        
      \node[] at (-0.2, 0.2){$B=(1,0)$};
         \node[] at (-0.2,-0.2){$C=(0,1)$};
     
 \draw [fill] (0,0) circle [radius=0.01];
 
  \draw [fill] (0.2,0.2) circle [radius=0.01];
 
   \draw [fill] (0.4,0) circle [radius=0.01];
 
      \draw[dashed line] (0.5, 0.3)-- (0.5,-0.3);

         \draw [fill] (0.6,-0.2) circle [radius=0.01];
 
      \draw [fill] (0.8,0.2) circle [radius=0.01];
 
   \draw [fill] (1.0,-0.2) circle [radius=0.01];
    \draw[dashed line] (1.1, 0.3)-- (1.1,-0.3);

      \draw [fill] (1.2,0) circle [radius=0.01];
 
      \draw [fill] (1.4,0.2) circle [radius=0.01];
 
   \draw [fill] (1.6,0) circle [radius=0.01];
 
      \end{tikzpicture} 
\end{center}
\caption{Type-$B$ presentation: $A\mapsto C, B\mapsto B, C\mapsto A$. The first and second columns (iii) $ABA \mapsto CBC$; the second and third columns (iv) $CBC \mapsto ABA$.}\label{figure:5}
\end{figure}
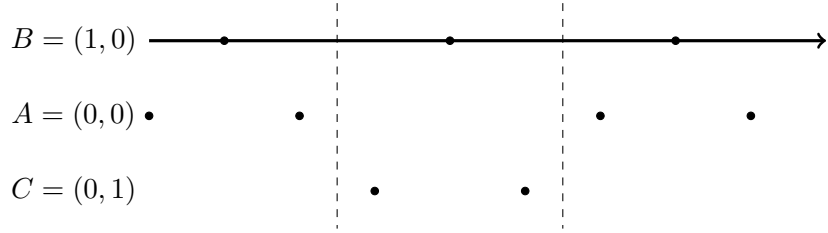

  \begin{figure}[h!]
\begin{center}
\begin{tikzpicture}[
    scale=5,
    axis/.style={very thick, ->},
    important line/.style={thick},
    dashed line/.style={dashed, thin},
    pile/.style={thick, ->, >=stealth', shorten <=2pt, shorten
    >=2pt},
    every node/.style={color=black}
    ]
    \draw[axis] (0,-0.2)  -- (1.8,-0.2) node(xline)[right]
        {};
        \node[] at (-0.2,0){$A=(0,0)$};
        
      \node[] at (-0.2, 0.2){$B=(1,0)$};
         \node[] at (-0.2,-0.2){$C=(0,1)$};
     
 \draw [fill] (0,0) circle [radius=0.01];
 
  \draw [fill] (0.2,-0.2) circle [radius=0.01];
 
   \draw [fill] (0.4,0) circle [radius=0.01];
 
      \draw[dashed line] (0.5, 0.3)-- (0.5,-0.3);

         \draw [fill] (0.6,0.2) circle [radius=0.01];
 
      \draw [fill] (0.8,-0.2) circle [radius=0.01];
 
   \draw [fill] (1.0,0.2) circle [radius=0.01];
    \draw[dashed line] (1.1, 0.3)-- (1.1,-0.3);

      \draw [fill] (1.2,0) circle [radius=0.01];
 
      \draw [fill] (1.4,-0.2) circle [radius=0.01];
 
   \draw [fill] (1.6,0) circle [radius=0.01];
 
      \end{tikzpicture} 
\end{center}
\caption{Type-$C$ presentation: $A\mapsto B, B\mapsto A, C\mapsto C$.  The first and second columns (v) $ACA \mapsto BCB$; the second and third columns (vi) $BCB \mapsto ACA$.}\label{figure:6}
\end{figure}
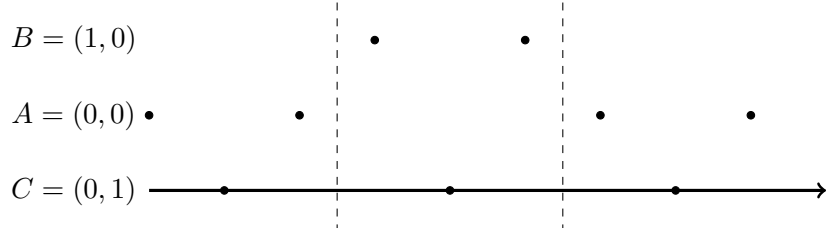

Before giving the proof we first explain (i)--(vi) via Figures \ref{figure:4}--\ref{figure:6}. Firstly,  (i) and (ii) are represented in Figure \ref{figure:4} which implies that if $(d_i)\in\US_\beta$ begins with $BAB$, then $(d_i)=(BABCAC)^\f$.  Secondly,   (iii) and (iv) are illustrated in Figure \ref{figure:5} which yields that if $(d_i)\in\US_\beta$ begins with $ABA$, then $(d_i)=(ABACBC)^\f$.  Finally,  (v) and (vi) are described in Figure \ref{figure:6} which gives that if $(d_i)\in\US_\beta$ begins with $ACA$,  then $(d_i)=(ACABCB)^\f$.  The names `Type $A$, Type-$B$' and `Type-$C$'  will be clarified in the next section.

 \begin{proof}
 Let $(d_i)\in\US_\beta$. Since the proof of (ii) is similar to (i), the proof of (v) is similar to (iii), and the proof of (vi) is similar to (iv), we only prove the lemma for cases (i), (iii) and (iv).
 
 (i). Suppose $d_{m+1}d_{m+2} d_{m+3}=BAB$ and $d_m\ne B$ (see Figure \ref{figure:4}). Then $d_m^1\ldots d_{m+3}^1=0101$. Note that  $\de(\beta)$ begins with $101000$. By Proposition \ref{prop:21} it follows that $d_{m+4}^1d_{m+5}^1 d_{m+6}^1=000$ which implies $d_{m+4}d_{m+5}d_{m+6}\in\set{A, C}^3$. Then by Lemma \ref{l42} we have only    two choices: either $d_{m+4}d_{m+5}d_{m+6}=ACA$ or $d_{m+4}d_{m+5}d_{m+6}=CAC$. 
 
 If $d_{m+4}d_{m+5}d_{m+6}=ACA$, then 
$d_{m+1}^\+ \ldots d_{m+6}^\+=101010.$ Observe that $\overline{\de(\beta)}$ begins with $010111$. 
 This implies that $d_{m+1}^\+=1$ and $d_{m+2}^\+d_{m+3}^\+\ldots\prec\overline{\de(\beta)}$. By Proposition \ref{prop:21} we have $(d_i)\notin\US_\beta$. Hence, $d_{m+4}d_{m+5}d_{m+6}=CAC$ as required.

 (iii). Suppose $d_{m+1}d_{m+2}d_{m+3}=ABA$ and $d_m\ne A$ (see Figure \ref{figure:5}). Then   $d_m^\+=1$ and  $d_{m+1}^\+ d_{m+2}^\+ d_{m+3}^\+=010$.  By Proposition \ref{prop:21} it follows that $d_{m+4}^\+ d_{m+5}^\+ d_{m+6}^\+=111$ which gives $d_{m+4}d_{m+5}d_{m+6}\in\set{B,C}^3$. Using Lemma \ref{l42} we   have only two choices: either $d_{m+4}d_{m+5}d_{m+6}=BCB$ or $d_{m+4}d_{m+5}d_{m+6}=CBC$.
 
 If $d_{m+4}d_{m+5}d_{m+6}=BCB$, then 
 $d_{m+1}^1\ldots d_{m+6}^1=010101,$ which implies $d_{m+1}^1=0$ and $d_{m+2}^1d_{m+3}^1\ldots\succ \de(\beta)$. 
By Proposition \ref{prop:21} this implies   $(d_i)\notin\us_\beta$. So, $d_{m+4}d_{m+5}d_{m+6}=CBC$ as required.  

(iv). Suppose $d_{m+1}d_{m+2}d_{m+3}=CBC$ and $d_m\ne C$ (see Figure \ref{figure:5} second column). Then
$ d_m^2d_{m+1}^2d_{m+2}^2d_{m+3}^2=0101.$
Since $\de(\beta)$ begins with $101000$, by Proposition \ref{prop:21} it follows that $d_{m+4}^2d_{m+5}^2d_{m+6}^2=000$ which yields $d_{m+4}d_{m+5}d_{m+6}\in\set{A,B}^3$. By Lemma \ref{l42}
we have only two choices: either $d_{m+4}d_{m+5}d_{m+6}=ABA$ or $d_{m+4}d_{m+5}d_{m+6}=BAB$. 

If $d_{m+4}d_{m+5}d_{m+6}=BAB$. Then 
$
d_{m+1}^1\ldots d_{m+6}^1=010101,
$ which implies $d_{m+1}^1=0$ and $d_{m+2}^1d_{m+3}^1\ldots \succ \de(\beta)$.
By Proposition \ref{prop:21} we have  $(d_i)\notin\us_\beta$. Hence, $d_{m+4}d_{m+5}d_{m+6}=ABA$. This completes the proof.  
 \end{proof}

 \begin{proof}[Proof of Proposition \ref{prop:41}]
Take $\beta\in(\beta_G,\beta_2]$. Note that $\de(\beta_G)=(100)^\f$ and $\de(\beta_2)=(101000)^\f$. Then  by Lemma \ref{lem:delta-beta} $\de(\beta)$ begins with $101000$. By Proposition \ref{prop:21} one can check that the sequences 
$(BAC)^\f$  and $(ABC)^\f$
 belong to $\US_\beta$. Note that the set-valued map $\beta\mapsto \US_\beta$ is increasing. It suffices to show that any sequence in $\US_{\beta_2}\setminus\US_{\beta_G}$ must end with   $(ABC)^\f$ or $(CBA)^\f$.
 
 Take $(d_i)\in\US_{\beta_2}\setminus\US_{\beta_G}$. Suppose on the contrary that $(d_i)$ ends with neither  $(ABC)^\f$ nor $(CBA)^\f$. Note by Theorem  \ref{th:sidorov} that $\US_{\beta_G}=\set{A^\f, B^\f, C^\f}$. Then by Lemma \ref{l42} there exists $m\ge 1$ such that 
 \[d_{m}\ne d_{m+1}\quad\textrm{and}\quad  d_{m+1}d_{m+2}d_{m+3}\in\set{ABA, BAB, ACA, CAC, BCB, CBC}.\]
Since the proof for different cases are similar, we only consider the case $d_{m+1}d_{m+2}d_{m+3}=ABA$ and $d_m\ne A$.
By Lemma \ref{lem:43} (iii) and (iv) it follows that $d_{m+1}d_{m+2}\ldots =(ABACBC)^\f$. This implies that 
\[
d_m^\+=1,\quad\textrm{and}\quad d_{m+1}^\+d_{m+2}^\+\ldots =(010111)^\f=\overline{\de(\beta_2)}.
\]
By Proposition \ref{prop:21} we have $(d_i)\notin\US_{\beta_2}$, leading to a contradiction. This completes the proof.
 \end{proof}

 \begin{proof}[Proof of Theorem \ref{th:S=D}]
 Note by Proposition \ref{lem:22}  that   for $\beta\in(1,3/2]$ we have $\u_\beta=\U_\beta$. Furthermore, by Theorem \ref{th:sidorov} we have $\US_{\beta_G}=\set{A^\f, B^\f, C^\f}$. Therefore, the theorem follows by the proof of Proposition \ref{prop:41}   that for any $\beta\in(\beta_G, \beta_2]$ the set $\US_{\beta}\setminus\US_{\beta_G}$ consists of all sequences of the form 
\[
d^n(BAC)^\f,\quad d^n(CAB)^\f;\quad d^n(CBA)^\f,\quad d^n(ABC)^\f;\quad d^n(ACB)^\f, \quad d^n(BCA)^\f,
\]
where $d\in\A$ and $n=0,1,2,\ldots$. 
 \end{proof}

 \section{Admissible words in $\US_{\beta_c}$}\label{sec:admissible}
In this section we will describe the admissible words   in sequences of  $\US_{\beta_c}$. Motivated by the previous section we will introduce three types of Thue-Morse words with respect to the alphabet $\A=\set{A, B, C}$. First we recall from Section \ref{s1}   the Thue-Morse type words $(\t_n)$ defined on  $\Om=\set{000, 001,100,101}$. 
 
 Let $\t_1=100$, and for $n\ge 1$ we set
 \[
 \t_{n+1}=\t_n^+\re({\t_n^+}),
 \]
 where $\re$ is the  block map defined by
 \[
 \re: \Om\ra \Om;\quad 000\mapsto 101,~ 001\mapsto 100,~ 100\mapsto 001,~101\mapsto 000.
 \] 
Here  for a word $\a=a_1\ldots a_m\in\Om^m$ we set $\re(\a)=\re(a_1)\ldots \re(a_m)$. Clearly, $\re\circ\re(\a)=\a$.
Observe that $\t_{n}$ ends with $0$ and has length $3\cdot 2^{n-1}$. Furthermore, $\t_{n+1}$ begins with $\t_n^+$ for all $n\ge 1$. This implies that the component-wise limit of the sequence $(\t_n)$ is well-defined, which is the Thue-Morse type sequence $(\la_i)$. So $(\la_i)$ begins with $\t_n^+$ for all $n\ge 1$. One can check that the first few terms of $(\la_i)$ are given by 
\[
101001000101\,000100101001\;000100101000101001000101.
\] 

Similarly, the component-wise limit of the sequence $(\re(\t_n))$ is also well-defined, denoted this limit by $(\ga_i)$. Then $(\ga_i)$ begins with $\re(\t_n^+)$ for all $n\ge 1$, and the first few terms of $(\ga_i)$ are given by
\[000100101000\,101001000100\; 101001000101000100101000.\]
 Therefore, for any $n\ge 1$ we have
$
\la_1\ldots \la_{3\cdot 2^{n-1}}=\t_n^+$ and $\ga_1\ldots \ga_{3\cdot 2^{n-1}}=\re(\t_n^+).$

  \begin{lemma}\label{l50}
For any $n\ge 1$ we have 
 \begin{eqnarray}
&&\ga_1\ldots \ga_{3\cdot2^{n-1}-i}\prec \la_{i+1}\ldots \la_{3\cdot2^{n-1}}\lle \la_1\ldots \la_{3\cdot 2^{n-1}-i} \label{eq:la-ga-ineq}\\
&&\ga_1\ldots \ga_{3\cdot2^{n-1}-i}\le \ga_{i+1}\ldots \ga_{3\cdot2^{n-1}}\prec \la_1\ldots \la_{3\cdot 2^{n-1}-i} \label{eq:la-ga-ineq-1}
 \end{eqnarray}
 {for all } $0\le i<3\cdot2^{n-1}.$
 \end{lemma}
 \begin{proof}
 Since the proof of (\ref{eq:la-ga-ineq-1}) is similar, here we only prove (\ref{eq:la-ga-ineq}).
 
 Let $(\tau_i)_{i=1}^\f=11010011\ldots$ be the  shifted classical Thue-Morse sequence (cf.~\cite{Komornik-Loreti-1998}, see also \cite{Glendinning_Sidorov_2001}).   By the definition of   $(\la_i)$  it follows   that $(\la_i)$ can  be obtained from $(\tau_i)$ by adding an extra zero between  $\tau_{2k+1}$ and $\tau_{2k+2}$ for each $k\ge 0$. Similarly, the sequence $(\ga_i)$ can be reconstructed from $\overline{(\tau_i)}$ by adding an extra zero between $\overline{\tau_{2k+1}}$ and $\overline{\tau_{2k+2}}$ for each $k\ge 0$, where  
 $\overline{0}=1$ and $\overline{1}=0$. In other words, for any $k\ge 0$ we have 
 \begin{equation}\label{eq:la-ga-1}
 \begin{split}
 &\la_{3k+1}=\tau_{2k+1},\quad \la_{3k+2}=0,\quad \la_{3k+3}=\tau_{2k+2},\\
 &\ga_{3k+1}=\overline{\tau_{2k+1}},\quad \ga_{3k+2}=0,\quad \ga_{3k+3}=\overline{\tau_{2k+2}}.
 \end{split}
 \end{equation}

Clearly, the inequalities in (\ref{eq:la-ga-ineq}) holds for $n=1$. Now let $n\ge 2$ and take $0\le i<3\cdot 2^{n-1}$.  We will prove (\ref{eq:la-ga-ineq}) by considering three cases.

(I). $i=3k$ with $0\le k<2^{n-1}$. Then by (\ref{eq:la-ga-1}) it follows that 
\[
\la_{3k+1}\ldots\la_{3\cdot 2^{n-1}}=\tau_{2k+1}0\tau_{2k+2}\;\tau_{2k+3}0\tau_{2k+4}\;\ldots\;\tau_{2^n-1}0\tau_{2^n}.
\]
Then (\ref{eq:la-ga-ineq}) follows by (\ref{eq:la-ga-1}) and    the property of $(\tau_i)$ that  
 \[
 \overline{\tau_1\ldots \tau_{2^{n}-j}}\prec\tau_{j+1}\ldots \tau_{2^{n}}\lle \tau_1\ldots \tau_{2^{n}-j} 
 \]for all $0\le j<2^{n}$ (cf.~\cite{Komornik-Loreti-1998}).
 
 (II). $i=3k+1$ with $0\le k<2^{n-1}$. Then by (\ref{eq:la-ga-1}) we have
 \[\la_{3k+2}\la_{3k+3}\la_{3k+4}=0\tau_{2k+2}\tau_{2k+3}=0\tau_{k+1}(1-\tau_{k+1}),\]
 where the last equality follows by using that $\tau_{2j}=\tau_j$ and $\tau_{2j+1}=1-\tau_{j}$ for all $j\ge 1$ (cf.~\cite{Allouche_Shallit_1999}). Thus, (\ref{eq:la-ga-ineq}) holds since $(\ga_i)$ begins with $000$ and $(\la_i)$ begins with $101$.
 
 (III). $i=3k+2$  with $0\le k<2^{n-1}$. Again by (\ref{eq:la-ga-1}) we have
 $
 \la_{3k+3}\la_{3k+4}\la_{3k+5}=\tau_{2k+2}\tau_{2k+3}0=\tau_{k+1}(1-\tau_{k+1})0.
 $
 By the same argument as in Case (II) we prove (\ref{eq:la-ga-ineq}). 
 \end{proof}

 Recall that $\de(\beta)$ is the quasi-greedy $\beta$-expansion of $1$. By Lemmas \ref{lem:delta-beta} and \ref{l50} it follows that  each  Thue-Morse type word $\t_n$  defines a unique base $\beta_n\in (1,2)$. 
 \begin{definition}
 \label{def:beta-n}
 Set $\beta_0=1$, and for $n\ge 1$ let $\beta_n\in(1, 2)$  such that  $\de(\beta_n)=(\t_n)^\f$.
 \end{definition}

In terms of  Definition \ref{def:beta-n}  we have  $\de(\beta_1)=(100)^\f$ and $\de(\beta_2)=(101000)^\f$. Then  $\beta_1=\beta_G$ and $\beta_2\approx 1.5385$ as defined  in  the previous section. 
 Furthermore,  by (\ref{eq:betaG-betaC}), Lemmas \ref{lem:delta-beta} and \ref{l50} the sequence $(\la_i)$ is indeed the quasi-greedy $\beta_c$-expansion of $1$, i.e., 
 \begin{equation}\label{eq:betac-delta}
 \de(\beta_c)=(\la_i)=\t_n^+\re(\t_n)\ldots=101001000101\;000100101001\ldots 
 \end{equation} 
 for any $n\ge 1$.   Observe that 
 \[\t_1\prec \t_2\prec\cdots\prec \t_n\prec\t_{n+1}\prec\cdots,\quad\textrm{and}\quad \t_n\nearrow (\la_i)\textrm{ as }n\ra\f.\]
 By Lemma \ref{lem:delta-beta} and  Definition \ref{def:beta-n} it follows that 
 \[
 1=\beta_0<\beta_1=\beta_G<\beta_2<\cdots<\beta_n<\beta_{n+1}<\cdots,\quad\textrm{and}\quad \beta_n\nearrow \beta_{c} \textrm{ as }n\ra\f.
 \]
 So these bases $(\beta_n)_{n=1}^\f$ form a partition of the interval $(1,\beta_c)$.
 
 In the following  we describe all possible admissible words in $\US_{\beta_c}$ which are constructed by  three types of Thue-Morse words in  $\A^*$. In the next section we will show that the set-valued map $\beta\mapsto \US_\beta$ is constant on each subinterval $(\beta_n, \beta_{n+1}]$ for all $n\ge 0$, and completely characterize the sets $\US_{\beta_n}$ and $\US_{\beta_c}$.
 
  \subsection{Type-$A$ Thue-Morse words in $\US_{\beta_c}$} 
Let $\Phi_A$ be the  bijective map on $\A$ defined by  
 \[
 \Phi_A: \A\ra\A;\quad  A\mapsto A, ~B\mapsto C, ~C\mapsto B.
 \]
Accordingly,  for a word $\d=d_1\ldots d_m\in\A^m$ we set $\Phi_A(\d)=\Phi_A(d_1)\ldots\Phi_A(d_m)$ as the word by concatenating  $\Phi_A(d_1), \ldots, \Phi_A(d_{m-1})$ and $\Phi_A(d_m)$. Clearly, $\Phi_A\circ\Phi_A(\d)=\d$. We point out that the name `Type-$A$' comes from   that $A$ is the fixed point of $\Phi_A$ (see Figure \ref{figure:4}).
 
\begin{definition}\label{def:type-a}
The \emph{Type-$A$ Thue-Morse words} $(\su_n)$ are defined as follows. Let $\su_0:=A, \su_1:=BAC$, and for $n\ge 1$ we set  
 \[
 \su_{n+1}:=\su_n^B\Phi_A(\su_n^B),
 \]
 where $\su_n^B$ is the word by changing the last digit of $\su_n$ to $B$. \end{definition}
By Definition \ref{def:type-a} one can verify  that  
 \begin{align*}
 &\su_2=BABCAC,\quad \su_3=BABCAB\,CACBAC,\\
 &\su_4=BABCAB\, CACBAB\;CACBACBABCAC.
 \end{align*}
Clearly, for each $n\ge 1$ the length of $\su_n$ is $3\cdot 2^{n-1}$ and each $\su_n$ ends with digit $C$.

Observe that for a given word $\d=d_1\ldots d_n\in\A^n$ the corresponding {coordinate words} 
\[\d^1:=d_1^1\ldots d_n^1, \quad \d^2:=d_1^2\ldots d_n^2\]
 are determined and they both belong to $\set{0,1}^n$.   
 The following lemma describes the relationship between the coordinate words of $\su_n$, $\Phi_A(\su_n)$ and the Thue-Morse type word $\t_n$.  
   
 \begin{lemma}\label{lem:51}
 For any $n\ge 1$ we have 
\begin{equation}\label{eq:un-property}
\su_n^1=\Phi_A(\su_n)^2= \t_n\quad{and}\quad  \su_n^2=\Phi_A(\su_n)^1 =\re({\t_n}).
\end{equation}
  \end{lemma}
 \begin{proof}
  We will prove this   by induction on $n$.   First we consider $n=1$. Then $\su_1=BAC$ and $\Phi_A(\su_1)=CAB$. This implies
 \[
 \su_1^1=\Phi_A(\su_1)^2=100=\t_1\quad\textrm{and}\quad \su_1^2=\Phi_A(\su_1)^1=001=\re(\t_1).
 \]
 So, (\ref{eq:un-property}) holds for $n=1$.
 
 Now suppose (\ref{eq:un-property}) holds for $n=k$, and we consider $n=k+1$. Note that  $\su_{k+1}=\su_k^B\Phi_A(\su_k^B)$ and $\su_k$ ends with digit $C$. Recall that $\t_{k+1}=\t_k^+\re({\t_k^+})$, and then $\re(\t_{k+1})=\re(\t_k^+)\t_k^+.$ By  using    the induction hypothesis  it follows that 
 \begin{align*}
 \su_{k+1}^1&=(\su_k^B)^1(\Phi_A(\su_k^B))^1=\t_k^+\re(\t_k^+)=\t_{k+1},\\
 \su_{k+1}^2&=(\su_k^B)^2(\Phi_A(\su_k^B))^2=\re(\t_k^+)\t_k^+=\re(\t_{k+1}).
 \end{align*}
 This proves (\ref{eq:un-property}) for $n=k+1$.  Hence, the lemma follows by induction. 
 \end{proof}
 
  In the next section  we will show that  for $\beta=\beta_c$ the Type-$A$ Thue-Morse words are all admissible in   $\US_{\beta_c}$. The following proposition states that once the block $\su_k^B$ or $\Phi_A(\su_k^B)$ occurs in a sequence of $\US_{\beta_c}$, then the next block of length $3\cdot 2^{k-1}$ is nearly determined. 
 \begin{proposition}\label{prop:betac-a}
 Let $(d_i)\in\US_{\beta_c}$. Then for any    $k\ge 1$ the following statements hold.
 \begin{enumerate}[{\rm(i)}]
 \item If $d_{m+1}\ldots d_{m+3\cdot2^{k-1}}=\su_k^B$ and $d_m\ne d_{m+1}$, then  
 \[d_{m+3\cdot2^{k-1}+1}\ldots d_{m+3\cdot2^k }\in\set{\Phi_A(\su_k^B), \Phi_A(\su_k)}.\]
 
 \item  If $d_{m+1}\ldots d_{m+3\cdot2^{k-1}}=\Phi_A(\su_k^B)$ and $d_m\ne d_{m+1}$, then   
 \[d_{m+3\cdot2^{k-1}+1}\ldots d_{m+3\cdot2^k}\in\set{\su_k^B, \su_k}.\]
 \end{enumerate}
 \end{proposition}
Since the proof of (ii) is similar to (i), we only prove  Proposition \ref{prop:betac-a} (i). We do this now by   induction on $k$. First we consider  the cases $k=1$ and $k=2$. 
\begin{lemma}\label{lem:betac-a1}
Let $(d_i)\in\US_{\beta_c}$.
\begin{enumerate}[{\rm(i)}]
\item If $d_{m+1}d_{m+2}d_{m+3}=\su_1^B$ and $d_m\ne d_{m+1}$, then  
 $d_{m+4}d_{m+5} d_{m+6}\in\set{\Phi_A(\su_1^B), \Phi_A(\su_1)}.$
 
 \item  If $d_{m+1}\ldots d_{m+6}=\su_2^B$ and $d_m\ne d_{m+1}$, then  
 $d_{m+7}\ldots d_{m+12}\in\set{\Phi_A(\su_2^B), \Phi_A(\su_2)}.$
 
 \end{enumerate}
\end{lemma}
\begin{proof}
First we prove (i).
Suppose  $d_{m+1}d_{m+2} d_{m+3}=\su_1^B=BAB$ and $d_m\ne B$. Then $d_{m}^1\ldots d_{m+3}^1=0101$. Recall from (\ref{eq:betac-delta}) that $\de(\beta_c)$ begins with $101001$. By Proposition \ref{prop:21} it follows that $d_{m+4}^1d_{m+5}^1=00$ which gives $d_{m+4}d_{m+5}\in\set{A, C}^2$. Then by Lemma \ref{l42} we have only two choices: either $d_{m+4}d_{m+5}=AC$ or $d_{m+4}d_{m+5}=CA$. 
 
  If $d_{m+4}d_{m+5}=AC$, then by Lemma \ref{l42} we have $d_{m+6}\in\set{A, B}$. We will prove in the following two cases that $(d_i)\notin\US_{\beta_c}$, and then conclude that $d_{m+4}d_{m+5}=CA$.
 \begin{itemize}
 \item Suppose $d_{m+6}=A$. Then, in view of Figure \ref{figure:10}, $d_{m+1}^\+=1$ and $d_{m+2}^\+\ldots d_{m+6}^\+=01010$, which implies that $d_{m+2}^\+d_{m+3}^\+\prec 010110\ldots=\overline{\de(\beta_c)}$. By Proposition \ref{prop:21} this  implies  $(d_i)\notin\US_{\beta_c}$. 
 
 \item Suppose $d_{m+6}=B$ (see Figure \ref{figure:10}). Then $d_{m}^1\ldots d_{m+6}^1=0101001$. Note by (\ref{eq:betac-delta}) that $\de(\beta_c)$ begins with $101001000$. Then by Proposition \ref{prop:21}  we have  $d_{m+7}^1d_{m+8}^1d_{m+9}^1=000$. So   by Lemma \ref{l42} it follows that \[d_{m+7}d_{m+8}d_{m+9}\in\set{ACA, CAC}.\]
   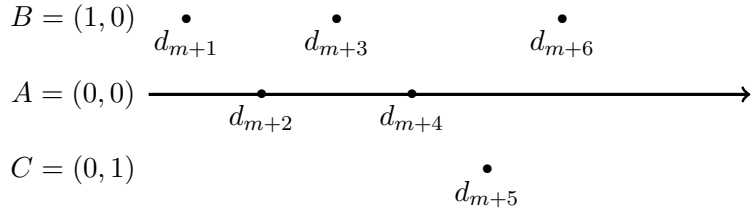
\begin{figure}[h!]
\begin{center}
\begin{tikzpicture}[
    scale=5,
    axis/.style={very thick, ->},
    important line/.style={thick},
    dashed line/.style={dashed, thin},
    pile/.style={thick, ->, >=stealth', shorten <=2pt, shorten
    >=2pt},
    every node/.style={color=black}
    ]
    \draw[axis] (-0.1,0)  -- (1.5,0) node(xline)[right]
        {};
        \node[] at (-0.3,0){$A=(0,0)$};
        
      \node[] at (-0.3, 0.2){$B=(1,0)$};
         \node[] at (-0.3,-0.2){$C=(0,1)$};
         
         \node[]at (0, 0.16){$\stackrel{\bullet}{d_{m+1}}$};
         
          \node[]at (0.2, -0.04){$\stackrel{\bullet }{d_{m+2}}$};
          
           \node[]at (0.4, 0.16){$\stackrel{\bullet}{ d_{m+3}}$};
     
  \node[]at (0.6, -0.04){$\stackrel{\bullet}{ d_{m+4}}$};
 
   \node[]at (0.8, -0.24){$\stackrel{\bullet}{ d_{m+5}}$};
   
    \node[]at (1.0, 0.16){$\stackrel{\bullet}{ d_{m+6}}$};

      \end{tikzpicture} 
\end{center}
\caption{The presentation of $d_{m+1}\ldots d_{m+6}=BABACB$.}\label{figure:10}
\end{figure}
   If $d_{m+7}d_{m+8}d_{m+9}=CAC$, then  $d_{m+4}^2\ldots d_{m+9}^2=010101$. By Proposition \ref{prop:21} and (\ref{eq:betac-delta}) this yields $(d_i)\notin\US_{\beta_c}$. 
   If 
   $d_{m+7}d_{m+8}d_{m+9}=ACA$, then   $d_{m+1}^\+=1$ and $d_{m+2}^\+\ldots d_{m+9}^\+=01011010$, which implies $d_{m+2}^\+d_{m+3}^\+\ldots \prec 01011011\ldots=\overline{\de(\beta_c)}$. By Proposition \ref{prop:21} this again gives $(d_i)\notin\US_{\beta_c}$.
 \end{itemize}
 
 Therefore, $d_{m+4}d_{m+5}=CA$. By Lemma \ref{l42} the next element  $d_{m+6}\in\set{B, C}$. Hence, 
 either $d_{m+4}d_{m+5}d_{m+6}=CAB=\Phi_A(\su_1)$ or $d_{m+4}d_{m+5}d_{m+6}=CAC=\Phi_A(\su_1^B)$.
This proves (i).
 
%
For (ii) we assume  $d_{m+1}\ldots d_{m+6}=\su_2^B=BABCAB$ (see Figure \ref{figure:11}) and $d_m\ne B$.
   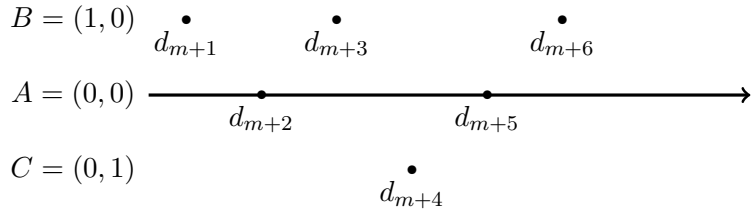
\begin{figure}[h!]
\begin{center}
\begin{tikzpicture}[
    scale=5,
    axis/.style={very thick, ->},
    important line/.style={thick},
    dashed line/.style={dashed, thin},
    pile/.style={thick, ->, >=stealth', shorten <=2pt, shorten
    >=2pt},
    every node/.style={color=black}
    ]
    \draw[axis] (-0.1,0)  -- (1.5,0) node(xline)[right]
        {};
        \node[] at (-0.3,0){$A=(0,0)$};
        
      \node[] at (-0.3, 0.2){$B=(1,0)$};
         \node[] at (-0.3,-0.2){$C=(0,1)$};
         
         \node[]at (0, 0.16){$\stackrel{\bullet}{d_{m+1}}$};
         
          \node[]at (0.2, -0.04){$\stackrel{\bullet }{d_{m+2}}$};
          
           \node[]at (0.4, 0.16){$\stackrel{\bullet}{ d_{m+3}}$};
     
  \node[]at (0.6, -0.24){$\stackrel{\bullet}{ d_{m+4}}$};
 
   \node[]at (0.8, -0.04){$\stackrel{\bullet}{ d_{m+5}}$};
   
    \node[]at (1.0, 0.16){$\stackrel{\bullet}{ d_{m+6}}$};

      \end{tikzpicture} 
\end{center}
\caption{The presentation of $d_{m+1}\ldots d_{m+6}=BABCAB$.}\label{figure:11}
\end{figure}
Then $d_m^1\ldots d_{m+6}^1=0101001$. Observe by (\ref{eq:betac-delta}) that $\de(\beta_c)$ begins with $101001000$. By Proposition \ref{prop:21} this implies that $d_{m+7}^1d_{m+8}^1 d_{m+9}^1=000$. So by Lemma \ref{l42} it follows that 
\[d_{m+7}d_{m+8}d_{m+9}\in\set{ACA, CAC}.\]
If $d_{m+7}d_{m+8}d_{m+9}=ACA$, then  $d_{m+4}^\+=1$ and $d_{m+5}^\+\ldots d_{m+9}^\+=01010$, which implies that $d_{m+5}^\+d_{m+6}^\+\ldots\prec 01011\ldots=\overline{\de(\beta_c)}$. By Proposition \ref{prop:21} this gives $(d_i)\notin\US_{\beta_c}$.

Therefore, $d_{m+7}d_{m+8}d_{m+9}=CAC=\Phi_A(\su_1^B)$. By a similar argument as in the proof of (i) one can show that the next block $d_{m+10}d_{m+11}d_{m+12}\in\set{\su_1^B, \su_1}=\set{BAB, BAC}$.    Hence, 
\[d_{m+7}\ldots d_{m+12}=CACBAB=\Phi_A(\su_2)\quad\textrm{or}\quad d_{m+7}\ldots d_{m+12}=CACBAC=\Phi_A(\su_2^B). \]
This establishes (ii).
\end{proof}

\begin{proof}[Proof of Proposition \ref{prop:betac-a}]
Since the proof of (ii) is similar, we only prove (i).
This will be done by induction on $k$. 
By Lemma \ref{lem:betac-a1}   it follows that the proposition holds for $k=1$ and $k=2$. Suppose the proposition holds for all $k\le n$ with $n\ge 2$, and we consider $k=n+1$. 

Suppose $d_{m+1}\ldots d_{m+3\cdot2^n}=\su_{n+1}^B$ and $d_m\ne d_{m+1}=B$. Then by Lemma \ref{lem:51} it follows that  $d_{m+1}^1\ldots d_{m+3\cdot2^n}^1=\t_{n+1}^+$.
Observe by (\ref{eq:betac-delta})  that $\de(\beta_c)=\t_{n+1}^+\re(\t_{n+1})\ldots$. Then  by Proposition \ref{prop:21}  and using   $d_m^1=0$ it follows that 
\begin{equation}\label{eq:betac-a1}
d^1_{m+3\cdot2^n+1}\ldots d^1_{m+3\cdot2^{n+1}}\lle \re(\t_{n+1})=\re(\t_{n-1}^+)\t_{n-1}\t_n^+,
\end{equation}
where the equality follows by using $\t_{i+1}=\t_i^+\re(\t_i^+)$ for any $i\ge 1$ and $\re^2(\a)=\a$ for any block $\a\in\Omega^*$.
Observe that 
\[
d_{m+1}\ldots d_{m+3\cdot2^n}=\su_{n+1}^B=\su_n^B\Phi_A(\su_n)=\su_n^B\Phi_A(\su_{n-1}^B)\su_{n-1}^B.
\]
Then $d_{m+3(2^{n-1}+2^{n-2})}=C$ is different from $d_{m+3(2^{n-1}+2^{n-2})+1}=B$. Furthermore, the block  $d_{m+3(2^{n-1}+2^{n-2})+1}\ldots d_{m+3\cdot2^n}=\su_{n-1}^B$. By the induction hypothesis it follows that the next block of length $3\cdot2^{n-2}$ is nearly determined, i.e., 
$
d_{m+3\cdot2^n+1}\ldots d_{m+3(2^n+2^{n-2})}\in\set{\Phi_A(\su_{n-1}^B), \Phi_A(\su_{n-1})}.
$
By (\ref{eq:betac-a1}) and   Lemma \ref{lem:51} it follows that 
\begin{equation}\label{eq:betac-a2}
d_{m+3\cdot2^n+1}\ldots d_{m+3(2^n+2^{n-2})}=\Phi_A(\su_{n-1}^B).
\end{equation}

Note that $d_{m+3\cdot2^n}=B$ is not equal to $d_{m+3\cdot2^n+1}=C$. Again, by (\ref{eq:betac-a2}) and the induction hypothesis it follows that 
$
d_{m+3(2^n+2^{n-2})+1}\ldots d_{m+3(2^n+2^{n-1})}\in\set{\su_{n-1}^B, \su_{n-1}}.
$
Using Lemma \ref{lem:51}, (\ref{eq:betac-a1}) and (\ref{eq:betac-a2}) we obtain that 
\begin{equation}\label{eq:betac-a3}
d_{m+3(2^n+2^{n-2})+1}\ldots d_{m+3(2^n+2^{n-1})}=\su_{n-1}.
\end{equation}
Therefore, by (\ref{eq:betac-a2}) and (\ref{eq:betac-a3}) it gives 
\[d_{m+3\cdot2^n+1}\ldots d_{m+3\cdot(2^n+2^{n-1})}=\Phi_A(\su_{n-1}^B)\su_{n-1}=\Phi_A(\su_{n}^B).\]

Again, using the induction hypothesis the next block of length $3\cdot2^{n-1}$ is nearly determined, i.e.,
$
d_{m+3(2^n+2^{n-1})+1}\ldots d_{m+3\cdot2^{n+1}}\in\set{\su_n, \su_n^B}.
$
Hence,  
\[
d_{m+3\cdot2^n+1}\ldots d_{m+3\cdot2^{n+1}}\in\set{\Phi_A(\su_n^B)\su_n, \Phi_A(\su_n^B)\su_n^B}=\set{\Phi_A(\su_{n+1}^B), \Phi_A(\su_{n+1})}.\]
This proves (i) for $k=n+1$. Hence, (i) follows by induction. 
\end{proof}

 \subsection{Type-$B$ and Type-$C$ Thue-Morse words in $\US_{\beta_c}$}
Similar to Definition \ref{def:type-a} we  define the Type-$B$ and Type-$C$ Thue-Morse words recursively.  First let $\Phi_B$ be defined by (see Figure \ref{figure:5})
 \[
 \Phi_B: \A\ra\A; \quad A\mapsto C,~ B\mapsto B, ~C\mapsto A.
 \]
 Accordingly, for a word $\d=d_1\ldots d_m\in\A^m$ we put $\Phi_B(\d)=\Phi_B(d_1)\ldots \Phi_B(d_m)$. 
 \begin{definition}\label{def:type-b} 
The \emph{Type-$B$ Thue-Morse words} $(\sv_n)$ are defined as follows. Let $\sv_0:=B, \sv_1:=CBA$, and for $n\ge 1$ we set
 \[
 \sv_{n+1}:=\sv_n^C\Phi_B(\sv_n^C),
 \]
 where $\sv_n^C$ is the word by changing the last digit of $\sv_n$ to $C$. 
 \end{definition}
 Then by Definition \ref{def:type-b} one can check that $\sv_2=CBCABA, \sv_3=CBCABC\,ABACBA,$ and
$\sv_4=CBCABC\, ABACBC\; ABACBACBCABA.
$
Furthermore, for each $n\ge 1$ the last digit of $\sv_n$ is $A$, and the length of $\sv_n$ equals $3\cdot 2^{n-1}$. 

Analogously,   
 let $\Phi_C$ be the  block map defined by (see Figure \ref{figure:6})
 \[
 \Phi_C: \A\ra\A; \quad A\mapsto B,~ B\mapsto A, ~C\mapsto C;
 \] and for a word $\d=d_1\ldots d_m\in\A^m$ we put $\Phi_C(\d)=\Phi_C(d_1)\ldots \Phi_C(d_m)$. 
\begin{definition}\label{def:type-c}
 The \emph{Type-$C$ Thue-Morse} words $(\sw_n)$  are defined as follows. Let $\sw_0:=C, \sw_1:=ACB$, and for $n\ge 1$ we set
 \[
 \sw_{n+1}:=\sw_n^A\Phi_C(\sw_n^A),
 \]
 where $\sw_n^A$ is the word by changing the last digit of $\sw_n$ to $A$. 
 \end{definition} Then by Definition \ref{def:type-c} we obtain that $\sw_2=ACABCB, \sw_3=ACABCA\,BCBACB$ and
$\sw_4= ACABCABCBACA\;BCBACBACABCB.
$
Furthermore, for each $n\ge 1$ the last digit of $\sw_n$ is $B$, and the length of $\sw_n$ equals $3\cdot 2^{n-1}$.
 
 The following lemma describes the coordinate words of $\sv_n, \Phi_B(\sv_n)$, $\sw_n$ and $\Phi_C(\sw_n)$.
  \begin{lemma}
 \label{lem:52}
 For any $n\ge 1$ we have 
 \begin{enumerate}[{\rm(i)}]
 \item $\sv_n^1=(010)^{2^{n-1}}$ and  $\sv_n^2=\t_n.$
 \item $\Phi_B(\sv_n)^1=(010)^{2^{n-1}}$ and $\Phi_B(\sv_n)^2=\re(\t_n)$.
 
  \item  $\sw_n^1=\re({\t_n})$ and $\sw_n^2=(010)^{2^{n-1}}$.
 \item $\Phi_C(\sw_n)^1=\t_n$ and $\Phi_C(\sw_n)^2=(010)^{2^{n-1}}$.

 \end{enumerate}
 \end{lemma}
\begin{proof}
Since the proofs of (iii) and (iv) are similar to (i) and (ii), we only prove the first two items.
By the definition of $\sv_n$ it follows that 
$
\sv_n^1=(010)^{2^{n-1}}$ and $\Phi_B(\sv_n)^1=(010)^{2^{n-1}}.$ 
So it suffices to prove 
\begin{equation}\label{eq:jh-1}
\sv_n^2=\t_n\quad\textrm{and}\quad \Phi_B(\sv_n)^2=\re(\t_n).
\end{equation}
 We will prove this by induction on $n$.

For $n=1$ we have $\sv_1=CBA$ and $\Phi_B(\sv_1)=ABC$. Then 
\[
\sv_1^2=100=\t_1\quad\textrm{and}\quad \Phi_B(\sv_1)^2=001=\re(\t_1).
\]
So, (\ref{eq:jh-1}) holds for $n=1$.
Now suppose (\ref{eq:jh-1}) holds for $n=k$, and we consider $n=k+1$. Note that 
\[
\sv_{k+1}=\sv_k^C\Phi_B(\sv_k^C),\quad \Phi_B(\sv_{k+1})=\Phi_B(\sv_k^C)\sv_k^C,
\]
where for the second equality we use $\Phi_B\circ\Phi_B(\d)=\d$ for any $\d\in\Omega^*$.
Moreover, for any $k\ge 1$ the word $\sv_k$ ends with digit $A$ and $\Phi_B(\sv_k)$ ends with digit  $C$. Therefore, by  the induction hypothesis it follows that 
\begin{align*}
\sv_{k+1}^2&=(\sv_k^C)^2(\Phi_B(\sv_k^C))^2=\t_k^+\re(\t_k^+)=\t_{k+1},\\
\Phi_B(\sv_{k+1})^2&=(\Phi_B(\sv_k^C))^2(\sv_k^C)^2=\re(\t_k^+)\t_k^+=\re(\t_{k+1}).
\end{align*}
This proves (\ref{eq:jh-1}) for $n=k+1$.  Hence, by induction we establish (i) and (ii).
\end{proof}

 For a word $\d=d_1\ldots d_m\in\A^m$ we write 
 \[ \d^\+:=d_1^\+\ldots d_m^\+=(d_1^1+d_1^2)\cdots(d_m^1+d_m^2).\]
  Then $\d^\oplus$ is a word of zeros and ones. Recall that for a word $\c=c_1\ldots c_m\in\set{0,1}^m$ its reflection is defined by $\overline{\c}=(1-c_1)\ldots (1-c_m)$.
 As a corollary of Lemmas \ref{lem:51} and \ref{lem:52}   we have the following.
 \begin{corollary}\label{cor:54}
 For any $n\ge 1$ we have 
 \begin{enumerate}[{\rm(i)}]
 \item $\su_n^\+=\Phi_A(\su_n)^\+=(101)^{2^{n-1}}$.
 
 \item $\sv_n^\+=\overline{\re(\t_n)}$ and $\Phi_B(\sv_n)^\+=\overline{\t_n}$.
 
 \item $\sw_n^\+=\overline{\t_n}$ and $\Phi_C(\sw_n)^\+=\overline{\re(\t_n)}$.
\end{enumerate}
 \end{corollary}
 
 Similar to Proposition \ref{prop:betac-a} we show that    the Type-$B$ and Type-$C$ Thue-Morse words are also forced  in $\US_{\beta_c}$.  
  \begin{proposition}
 \label{prop:betac-b}
 Let $(d_i)\in\US_{\beta_c}$. Then for any   $k\ge 1$ the following statements hold.
 \begin{enumerate}[{\rm(i)}]
 \item If   $d_{m+1}\ldots d_{m+3\cdot2^{k-1}}=\sv_k^C$ and $d_m\ne d_{m+1}$, then     
 \[d_{m+3\cdot2^{k-1}+1}\ldots d_{m+3\cdot2^k}\in\set{\Phi_B(\sv_k^C), \Phi_B(\sv_k)}.\]
    \item If  $d_{m+1}\ldots d_{m+3\cdot2^{k-1}}=\Phi_B(\sv_k^C)$ and $d_m\ne d_{m+1}$, then 
    \[d_{m+3\cdot2^{k-1}+1}\ldots d_{m+3\cdot2^k}=\set{\sv_k^C, \sv_k}.\]
     \item If   $d_{m+1}\ldots d_{m+3\cdot2^{k-1}}=\sw_k^A$ and $d_m\ne d_{m+1}$, then     
 \[d_{m+3\cdot2^{k-1}+1}\ldots d_{m+3\cdot2^k}\in\set{\Phi_C(\sw_k^A), \Phi_C(\sw_k)}.\]
    \item If  $d_{m+1}\ldots d_{m+3\cdot2^{k-1}}=\Phi_C(\sw_k^A)$ and $d_m\ne d_{m+1}$, then 
    \[d_{m+3\cdot2^{k-1}+1}\ldots d_{m+3\cdot2^k}=\set{\sw_k^A, \sw_k}.\]
 \end{enumerate}
 \end{proposition}
 Since the proofs for  different items in Proposition \ref{prop:betac-b} are similar, we only prove the first item. This will be done  by induction on $k$. First we  consider  the cases $k=1$ and $k=2$.

 \begin{lemma}\label{lem:betac-b1}
  Let $(d_i)\in\US_{\beta_c}$.  
  \begin{enumerate}[{\rm(i)}]
 \item If   $d_{m+1}d_{m+2} d_{m+3}=\sv_1^C$ and $d_m\ne d_{m+1}$, then     
 $d_{m+4}d_{m+5} d_{m+6}\in\set{\Phi_B(\sv_1^C), \Phi_B(\sv_1)}.$
 
 \item If   $d_{m+1}\ldots d_{m+6}=\sv_2^C$ with $d_m\ne d_{m+1}$, then     
 $d_{m+7}\ldots d_{m+12}\in\set{\Phi_B(\sv_2^C), \Phi_B(\sv_2)}.$

 \end{enumerate}
 \end{lemma}
 \begin{proof}
Suppose $d_{m+1}d_{m+2}d_{m+3}=\sv_1^C=CBC$ and $d_m\ne C$. Then $d_{m}^2\ldots d_{m+3}^2=0101$. Note by (\ref{eq:betac-delta}) that $\de(\beta_c)$ begins with $10100$. By Proposition \ref{prop:21} it follows that $d_{m+4}^2d_{m+5}^2=00$, which implies $d_{m+4}d_{m+5}\in\set{A, B}^2$. Then by Lemma \ref{l42} we have only two choices: either $d_{m+4}d_{m+5}=BA$ or $d_{m+4}d_{m+5}=AB$. We claim that the first case can not happen.

Suppose  $d_{m+4}d_{m+5}=BA$. Then by Lemma \ref{l42} we have $d_{m+6}\in\set{B, C}$. If $d_{m+6}=B$, then in view of Figure \ref{figure:12} we have $d_{m+1}^1\ldots d_{m+6}^1=010101$. Then $d_{m+1}^1=0$ and $d_{m+2}d_{m+3}\ldots\succ 10100\ldots=\de(\beta_c)$. So $(d_i)\notin\US_{\beta_c}$ by Proposition \ref{prop:21}.

   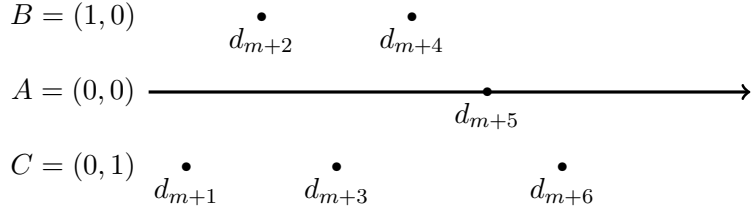
\begin{figure}[h!]
\begin{center}
\begin{tikzpicture}[
    scale=5,
    axis/.style={very thick, ->},
    important line/.style={thick},
    dashed line/.style={dashed, thin},
    pile/.style={thick, ->, >=stealth', shorten <=2pt, shorten
    >=2pt},
    every node/.style={color=black}
    ]
    \draw[axis] (-0.1,0)  -- (1.5,0) node(xline)[right]
        {};
        \node[] at (-0.3,0){$A=(0,0)$};
        
      \node[] at (-0.3, 0.2){$B=(1,0)$};
         \node[] at (-0.3,-0.2){$C=(0,1)$};
         
         \node[]at (0, -0.24){$\stackrel{\bullet}{d_{m+1}}$};
         
          \node[]at (0.2, 0.16){$\stackrel{\bullet }{d_{m+2}}$};
          
           \node[]at (0.4, -0.24){$\stackrel{\bullet}{ d_{m+3}}$};
     
  \node[]at (0.6, 0.16){$\stackrel{\bullet}{ d_{m+4}}$};
 
   \node[]at (0.8, -0.04){$\stackrel{\bullet}{ d_{m+5}}$};
   
    \node[]at (1.0, -0.24){$\stackrel{\bullet}{ d_{m+6}}$};

      \end{tikzpicture} 
\end{center}
\caption{The presentation of $d_{m+1}\ldots d_{m+6}=CBCBAC$.}\label{figure:12}
\end{figure}
 If $d_{m+6}=C$, then in view of  Figure \ref{figure:12} we have   $d_{m}^2\ldots d_{m+6}^2=0101001$. Since by (\ref{eq:betac-delta})  that $\de(\beta_c)$ begins with $101001000$, by Proposition \ref{prop:21}  it follows that $d_{m+7}^2d_{m+8}^2d_{m+9}^2=000$. This gives $d_{m+7}d_{m+8}d_{m+9}\in\set{A,B}^3$. Then by Lemma \ref{l42} we have $d_{m+7}d_{m+8}d_{m+9}\in\set{BAB, ABA}$.
\begin{itemize}
\item If $d_{m+7}d_{m+8}d_{m+9}=BAB$, then in view of Figure \ref{figure:12} we have $d_{m}^1\ldots d_{m+9}^1=010100101$. This implies $d_m^1=0$ and $d_{m+1}^1d_{m+2}^1\ldots\succ \de(\beta)$. Thus $(d_i)\notin\US_{\beta_c}$ by Proposition \ref{prop:21}.

\item If $d_{m+7}d_{m+8}d_{m+9}=ABA$, then  $d_{m+4}^\+=1$ and $d_{m+5}^\+\ldots d_{m+9}^\+=01010$, which implies $d_{m+5}^\+d_{m+6}^\+\ldots \prec \overline{\de(\beta_c)}=01011\ldots$. Again, by Proposition \ref{prop:21}   we have $(d_i)\notin\US_{\beta_c}$.
\end{itemize}

Therefore, $d_{m+4}d_{m+5}=AB$. So by Lemma \ref{l42} we have $d_{m+6}\in\set{A, C}$. Hence, 
\[
d_{m+4}d_{m+5}d_{m+6}\in\set{ABA, ABC}=\set{\Phi_B(\sv_1^C), \Phi_B(\sv_1)}
\]
as required.

  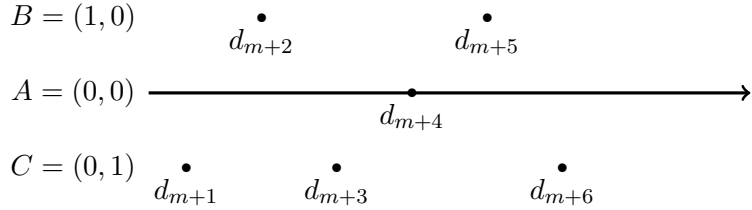
\begin{figure}[h!]
\begin{center}
\begin{tikzpicture}[
    scale=5,
    axis/.style={very thick, ->},
    important line/.style={thick},
    dashed line/.style={dashed, thin},
    pile/.style={thick, ->, >=stealth', shorten <=2pt, shorten
    >=2pt},
    every node/.style={color=black}
    ]
    \draw[axis] (-0.1,0)  -- (1.5,0) node(xline)[right]
        {};
        \node[] at (-0.3,0){$A=(0,0)$};
        
      \node[] at (-0.3, 0.2){$B=(1,0)$};
         \node[] at (-0.3,-0.2){$C=(0,1)$};
         
         \node[]at (0, -0.24){$\stackrel{\bullet}{d_{m+1}}$};
         
          \node[]at (0.2, 0.16){$\stackrel{\bullet }{d_{m+2}}$};
          
           \node[]at (0.4, -0.24){$\stackrel{\bullet}{ d_{m+3}}$};
     
  \node[]at (0.6, -0.04){$\stackrel{\bullet}{ d_{m+4}}$};
 
   \node[]at (0.8, 0.16){$\stackrel{\bullet}{ d_{m+5}}$};
   
    \node[]at (1.0, -0.24){$\stackrel{\bullet}{ d_{m+6}}$};

      \end{tikzpicture} 
\end{center}
\caption{The presentation of $d_{m+1}\ldots d_{m+6}=CBCABC$.}\label{figure:14}
\end{figure}

%
 For (ii) let  $d_{m+1}\ldots d_{m+6}=\sv_2^C=CBCABC$ (see Figure \ref{figure:14}) and $d_m\ne C$.
   Then $d_{m}^2\ldots d_{m+6}^2=0101001$. By (\ref{eq:betac-delta}) and Proposition \ref{prop:21} it follows that 
$d_{m+7}^2d_{m+8}^2d_{m+9}^2=000$. Using Lemma \ref{l42} we have $d_{m+7}d_{m+8}d_{m+9}\in\set{BAB, ABA}$.

Suppose $d_{m+7}d_{m+8}d_{m+9}=BAB$. Then in view of Figure \ref{figure:14} we have $d_{m+4}^1\ldots d_{m+9}^1=010101$.  This implies $(d_i)\notin\US_{\beta_c}$ by (\ref{eq:betac-delta}) and Proposition \ref{prop:21}.

Therefore, $d_{m+7}d_{m+8}d_{m+9}=ABA=\Phi_B(\sv_1^C)$. By the same argument as in the proof of (i) we can deduce that the next block
$d_{m+10}d_{m+11}d_{m+12}\in\set{\sv_1^C, \sv_1}=\set{CBC, CBA}.$
Whence, 
\[
d_{m+7}\ldots d_{m+12}\in\set{ABACBC, ABACBA}=\set{\Phi_B(\sv_2), \Phi_B(\sv_2^C)}.
\]
This completes the proof.
 \end{proof}
 
 \begin{proof}[Proof of Proposition \ref{prop:betac-b}]
 As remarked before  the proofs for different items are similar. So in the following we only prove (i) by using the same strategy as in the proof of Proposition \ref{prop:betac-a}. We do this now by induction on $k$. First by Lemma \ref{lem:betac-b1} it follows that  (i) holds for $k=1$ and $k=2$. Now suppose (i) holds for all $k\le n$ with $n\ge 2$, and we consider $k=n+1$.
 
  Suppose $d_{m+1}\ldots d_{m+3\cdot2^n}=\sv_{n+1}^C$ and $d_m\ne d_{m+1}=C$. Then $d_{m}^2=0$, and by Lemma \ref{lem:52} we have $d_{m+1}^2\ldots d_{m+3\cdot2^n}^2=(\sv_{n+1}^C)^2=\t_{n+1}^+$. Note by (\ref{eq:betac-delta}) that $\de(\beta_c)$ begins with $\t_{n+1}^+\re(\t_{n+1})$. By Proposition \ref{prop:21} it follows that 
 \begin{equation}\label{eq:betac-b}
 d_{m+3\cdot2^n+1}^2\ldots d_{m+3\cdot2^{n+1}}^2\lle \re(\t_{n+1})=\re(\t_{n-1}^+)\t_{n-1}\t_n^+.
 \end{equation}
 Observe that $d_{m+1}\ldots d_{m+3\cdot2^n}=\sv_{n+1}^C=\sv_n^C\Phi_B(\sv_{n-1}^C)\sv_{n-1}^C$. Then by the induction hypothesis it follows that  
 $d_{m+3\cdot2^n+1}\ldots d_{m+3(2^n+2^{n-2})}\in\set{\Phi_B(\sv_{n-1}^C), \Phi_B(\sv_{n-1})}.$
  By (\ref{eq:betac-b}) and Lemma \ref{lem:52} it follows that 
 \begin{equation}\label{eq:betac-b1}
 d_{m+3\cdot2^n+1}\ldots d_{m+3(2^n+2^{n-2})}=\Phi_B(\sv_{n-1}^C).
 \end{equation}
 Again by the induction hypothesis the next block of length $3\cdot2^{n-2}$ is  nearly determined, i.e., 
 $d_{m+3(2^n+2^{n-2})+1}\ldots d_{m+3(2^n+2^{n-1})}\in\set{\sv_{n-1}, \sv_{n-1}^C}.$ In terms of Lemma \ref{lem:52}, by (\ref{eq:betac-b}) and (\ref{eq:betac-b1}) we conclude that 
 \begin{equation}\label{eq:betac-b2}
 d_{m+3(2^n+2^{n-2})+1}\ldots d_{m+3(2^n+2^{n-1})}=\sv_{n-1}.
 \end{equation}
 Therefore, by (\ref{eq:betac-b1}) and (\ref{eq:betac-b2}) we have $d_{m+3\cdot2^n+1}\ldots d_{m+3(2^n+2^{n-1})}=\Phi_B(\sv_{n-1}^C)\sv_{n-1}=\Phi_B(\sv_{n}^C)$. Using the induction hypothesis the next block of length $3\cdot2^{n-1}$ is  nearly determined, i.e., $d_{m+3(2^n+2^{n-1})+1}\ldots d_{m+3\cdot2^{n+1}}\in\set{\sv_{n}^C, \sv_n}$. Hence, 
 \[
 d_{m+3\cdot2^n+1}\ldots d_{m+3\cdot2^{n+1}}\in\set{\Phi_B(\sv_n^C)\sv_n^C, \Phi_B(\sv_n^C)\sv_n}=\set{\Phi_B(\sv_{n+1}), \Phi_B(\sv_{n+1}^C)}.
 \]
 This proves (i) for $k=n+1$. Hence, by induction this completes the proof.
 \end{proof}

\section{Proof of Theorem \ref{thm:3}}\label{s6}
 In this section we will show that the number $\beta_{c}$ defined in (\ref{eq:betaG-betaC}) is transcendental and it is the critical base for the intrinsic univoque set $\U_\beta$, and then prove Theorem \ref{thm:3}.  
Motived by the work of Allouche and Cosnard \cite{Allouche_Cosnard_2000} (see also, \cite{Kong_Li_2015})  we first prove the transcendental of $\beta_c$ by using the following well known result from Mahler \cite{Mahler_1976}.
\begin{theorem}[Mahler, 1976]\label{th:mahler}
If $z$ is an algebraic number in the open unit disc, then the number 
\[
Z:=\sum_{i=1}^\f\tau_i z^i
\]
is transcendental,
where $(\tau_i)_{i=0}^\f$ is the classical Thue-Morse sequence.
\end{theorem}

\begin{proposition}\label{prop:transcendental}
$\beta_c$ is transcendental. 
\end{proposition}
\begin{proof}
Note by (\ref{eq:la-ga-1})  that the quasi-greedy expansion $\de(\beta_c)=(\la_i)$ satisfies 
\[\la_{3n+1}=\tau_{2n+1},\quad\la_{3n+2}=0\quad\textrm{and}\quad \la_{3n+3}=\tau_{2n+2}\]for all $n\ge 0$,
where $(\tau_i)_{i=0}^\f$ is the classical Thue-Morse seuqnece. Using that $\tau_0=0$ and $\tau_{2n+1}=1-\tau_n, \tau_{2n+2}=\tau_{n+1}$ for any $n\ge 0$ it follows that  
\begin{align*}
1=\sum_{i=1}^\f\frac{\la_i}{\beta^i}=\sum_{n=0}^\f\frac{\tau_{2n+1}}{\beta^{3n+1}}+\sum_{n=0}^\f\frac{\tau_{2n+2}}{\beta^{3n+3}}&=\sum_{n=0}^\f\frac{1-\tau_n}{\beta^{3n+1}}+\sum_{n=0}^\f\frac{\tau_{n+1}}{\beta^{3n+3}}\\
&=\frac{\beta^2}{\beta^3-1}-\frac{1}{\beta}\sum_{n=1}^\f\frac{\tau_n }{\beta^{3n}}+\sum_{n=1}^\f\frac{\tau_n}{\beta^{3n}}.
\end{align*} 
By rearrangement we have
\[
\sum_{n=1}^\f\frac{\tau_n}{\beta^{3n}}=\frac{\beta(\beta^3-\beta^2-1)}{(\beta-1)(\beta^3-1)}.
\]
If $\beta\in(1,2)$ is algebraic, then  by Theorem \ref{th:mahler} the lefthand side of the above equation is transcendental, while the righthand side is algebraic, leading to a contradiction. So $\beta_c$ is transcendental. 
\end{proof}

Recall from Definition \ref{def:beta-n} that the bases $\beta_n$ strictly increases to $\beta_c$ as $n\to\f$. In the following we show that for each $n\ge 0$ the periodic sequences $(\su_n)^\f, (\sv_n)^\f$ and $(\sw_n)^\f$ are all contained in $\US_{\beta_c}$. 
\begin{lemma}\label{prop:betan}
Let $n\ge 1$ and let $\beta>\beta_n$. Then the following periodic sequences
\[
(\su_k)^\f, ~(\sv_k)^\f,~(\sw_k)^\f,\quad (\Phi_A(\su_k))^\f,~(\Phi_B(\sv_k))^\f,~(\Phi_C(\sw_k))^\f \quad \textrm{with }0\le k\le n,
\] all belong to $\US_{\beta}$.
\end{lemma}
 \begin{proof}
Note by Theorem \ref{th:sidorov} and Proposition \ref{prop:41} that the lemma holds for $n=0$ and $n=1$. Now we consider $n\ge 2$ and $\beta>\beta_n$. Since the set-valued map $\beta\mapsto\US_{\beta}$ is increasing, and $\su_k=\su_{k-1}^B\Phi_A(\su_{k-1}^B)$, $\sv_k=\sv_{k-1}^C\Phi_B(\sv_{k-1}^C)$ and $\sw_k=\sw_{k-1}^A\Phi_C(\sw_{k-1}^C)$ for any $k\ge 2$, it suffices to prove that   
\[
(\su_k)^\f, ~(\sv_k)^\f,~(\sw_k)^\f~\in\US_{\beta}\quad\textrm{for all }2\le k\le n.
\]

Take $k\in\set{2,3,\ldots, n}$. Since $\beta>\beta_n$, by Definition \ref{def:beta-n}  and Lemma \ref{lem:delta-beta} it follows that 
\begin{equation}\label{eq:delta-beta-n}
\de(\beta)\succ \de(\beta_n)=(\t_n)^\f\lge (\t_k)^\f\lge(\t_2)^\f=(101000)^\f.
\end{equation}
Observe by Lemma \ref{lem:51} and Corollary \ref{cor:54} that 
\[\su_k^1=\t_k,\quad \su_k^2=\re(\t_k),\quad \su_k^\+=(101)^{2^{k-1}}.\]
 Then by Lemma \ref{l50} and (\ref{eq:delta-beta-n}) it follows that  
\begin{align*} 
& \si^i((\su_k^1)^\f)=\si^i((\t_k)^\f)\lle(\t_k)^\f\prec \de(\beta),\\
&\si^i((\su_k^2)^\f)=\si^i((\re(\t_k))^\f)\lle(\t_k)^\f\prec \de(\beta),\\
&\si^i(\overline{(\su_k^\+)^\f})=\si^i((010)^\f)\prec\de(\beta).
\end{align*}
Hence, by Proposition \ref{prop:21} we conclude that $(\su_k)^\f\in\US_{\beta}$.

Similarly, by Lemma \ref{lem:52} and Corollary \ref{cor:54} we have 
\begin{align*}
&\sv^1_k=(010)^{2^{k-1}},\quad \sv_k^2=\t_k,\quad \sv_k^\+=\overline{\re(\t_k)};\\
&\sw^1_k=\re(\t_k),\quad \sw_k^2=(010)^{2^{k-1}},\quad \sw_k^2=\overline{\t_k}.
\end{align*} 
Then by  (\ref{eq:delta-beta-n}), Lemma \ref{l50} and  Proposition \ref{prop:21} we can deduce that  $(\sv_k)^\f, (\sw_k)^\f\in\US_{\beta}$.
 \end{proof}
 
 In the following lemma we show that the sequences $(\su_n)^\f, (\sv_n)^\f$ and $(\sw_n)^\f$ are forbidden in $\US_{\beta_{n}}$. So the range for the parameter $\beta$ in the previous lemma is critical. 
 \begin{lemma}
 \label{l611}
 Let $(d_i)\in\US_{\beta_c}$. If  there exits $m\in\N$ such that    $d_m\ne d_{m+1}$  and  
 \[
 d_{m+1}\ldots d_{m+2^{n-2}3}\in\set{\su_{n-1}^B,  \Phi_A(\su_{n-1}^B),  \sv_{n-1}^C,  \Phi_B(\sv_{n-1}^C), \sw_{n-1}^A,  \Phi_C(\sw_{n-1}^A)},
 \]
 then $(d_i)\notin\US_{\beta_n}$.
 \end{lemma}
 \begin{proof}
 Since the proofs for different cases are similar, without loss of generality we assume on the contrary that  $d_{m+1}\ldots d_{m+2^{n-2}3}=\su_{n-1}^B$ and $d_m\ne d_{m+1}=B$. By Proposition \ref{prop:betac-a} it follows that the next block of length $2^{n-2}3$ is nearly determined:
 \[
 d_{m+1}\ldots d_{m+2^{n-1}3}\in\set{\su_n, \su_{n}^B}.
 \]
 If $ d_{m+1}\ldots d_{m+2^{n-1}3}=\su_{n}^B$, then by Lemma \ref{lem:51} we have $(\su_n^B)^1=\t_n^+$. This implies \[d_{m+1}^1 d_{m+2}^1\ldots=\t_n^+\ldots\succ(\t_n)^\f=\de(\beta_n).\] Since $d_m^1=0$, by Proposition \ref{prop:21} this implies that $(d_i)\notin\US_{\beta_n}$.

So $d_{m+1}\ldots d_{m+2^{n-1}3}=\su_{n}=\su_{n-1}^B\Phi_A({\su_{n-1}^B})$. Then  by Proposition \ref{prop:betac-a} the next block of length $2^{n-2}3$ is nearly determined, and thus
 \[
 d_{m+2^{n-2}3+1}\ldots d_{m+(2^{n-1}+2^{n-2})3}\in\set{\Phi_A(\su_{n-1}^B)\su_{n-1}^B, \Phi_A(\su_{n-1}^B)\su_{n-1}}.
 \]
 If $d_{m+2^{n-2}3+1}\ldots d_{m+(2^{n-1}+2^{n-2})3}=\Phi_A(\su_{n-1}^B)\su_{n-1}=\Phi_A(\su_n^B)$, then by   Lemma \ref{lem:51} we have $(\Phi_A(\su_n^B))^2=\t_n^+$. This implies $(d_i)\notin\US_{\beta_n}$ by Proposition \ref{prop:21}. Therefore, 
 \[
 d_{m+1}\ldots d_{m+(2^{n-1}+2^{n-2})3}=\su_{n-1}^B\Phi_A(\su_{n-1}^B)\su_{n-1}^B.
 \]
 
Repeating the above arguments it follows that $d_{m+1}d_{m+2}\ldots=(\su_{n-1}^B\Phi_A(\su_{n-1}^B))^\f=(\su_n)^\f$. Then    by Lemma \ref{lem:51} we have  
 \[
 d_m^1=0\quad\textrm{and}\quad d_{m+1}^1d_{m+2}^1\ldots=(\t_n)^\f=\de(\beta_n),
 \] 
 which again gives $(d_i)\notin\US_{\beta_n}$ by Proposition \ref{prop:21}.  This completes the proof.
 \end{proof}
 
The next  result  is a generalization of Proposition \ref{prop:41}. Since $\beta_n\nearrow \beta_c$ as $n\ra\f$, as a consequence of the following proposition  we establish Theorem \ref{thm:3} (i).
 \begin{proposition}\label{prop:54}
 Let $n\ge 0$. Then for any $\beta\in(\beta_n, \beta_{n+1}]$ we have 
 $\US_\beta=\US_{\beta_{n+1}}.$
  Furthermore,  any sequence in $\US_{\beta_{n+1}}$ must  end in
 \[\bigcup_{k=0}^n\set{(\su_k)^\f, (\sv_k)^\f, (\sw_k)^\f}.\]
 \end{proposition}
   \begin{proof}
 We will prove the proposition by induction on $n$. By Theorem \ref{th:sidorov} and Proposition \ref{prop:41} it follows that the theorem holds for $n=0$ and $n=1$. Now take $n\ge 2$ and $\beta\in(\beta_n, \beta_{n+1}]$. By Lemma \ref{prop:betan} it follows that the periodic sequences $(\su_k)^\f, (\sv_k)^\f$ and $(\sw_k)^\f$ with $0\le k\le n$ all belong to $\US_{\beta_{n+1}}$.
 
Let $(d_i)\in\US_{\beta}\setminus\US_{\beta_2}$. Observe by Lemma \ref{l42} that any block of the form `$cdd$' is forbidden in sequences of $(d_i)$. By Theorem \ref{th:sidorov} and Proposition \ref{prop:41} it follows that $(d_i)$ must contain one of the following blocks:
\begin{align*}
&\su_1^B=BAB,\quad\Phi_A(\su_1^B)=CAC;\\
& \sv_1^C=CBC,\quad\Phi_B(\sv_1^C)=ABA;\\
& \sw_1^A=ACA,\quad \Phi_C(\sw_1^A)=BCB.
\end{align*}
Since the proofs for different cases are similar,   we may assume   $d_{m+1}d_{m+2}d_{m+3}=\su_1^B$ and $d_m\ne d_{m+1}$ for some smallest integer $m\ge 1$. By Proposition \ref{prop:betac-a} it follows that the next block of length $3$ is nearly determined, i.e., $d_{m+4}d_{m+5}d_{m+6}\in\set{\Phi_A(\su_1^B), \Phi_A(\su_1)}$. Then    
\[
d_{m+1}\ldots d_{m+6}=\su_1^B\Phi_A(\su_1^B)\quad\textrm{or}\quad d_{m+1}\ldots d_{m+6}=\su_1^B\Phi_A(\su_1)=\su_2^B.
\]
Again, by Proposition \ref{prop:betac-a} we obtain the following recursive relation: for any $\ell\ge 1$
\begin{itemize}
\item if $d_{s+1}\ldots d_{s+3\cdot2^{\ell-1}}=\su_\ell^B$ and $d_s\ne d_{s+1}$, then either  $d_{s+3\cdot2^{\ell-1}+1}\ldots d_{s+3\cdot2^\ell}=\Phi_A(\su_\ell^B)$ or $d_{s+1}\ldots d_{s+3\cdot2^\ell}=\su_{\ell+1}^B$;

\item if $d_{s+1}\ldots d_{s+3\cdot2^{\ell-1}}=\Phi_A(\su_\ell^B)$ and $d_s\ne d_{s+1}$, then either  $d_{s+3\cdot2^{\ell-1}+1}\ldots d_{s+3\cdot2^\ell}=\su_\ell^B$ or $d_{s+1}\ldots d_{s+3\cdot2^\ell}=\Phi_A(\su_{\ell+1}^B)$.
\end{itemize}
Applying this recursive relation and using Lemma \ref{l611} it follows that the  sequence $d_{m+1}d_{m+2}\ldots$ is nearly determined, and it eventually    ends with $(\su_{k-1}^B\Phi_A(\su_{k-1}^B))^\f=(\su_k)^\f$ for some $2\le k\le n$. 
 
Observe that our proof does not depend on the choice of $\beta\in(\beta_n, \beta_{n+1}]$. So the set-valued map $\beta\mapsto\US_\beta$ is constant  in $(\beta_n, \beta_{n+1}]$, i.e., $\US_{\beta}=\US_{\beta_{n+1}}$ for any $\beta\in(\beta_n, \beta_{n+1}]$. This completes the proof.
 \end{proof}

Motivated by  the recursive relation  in  the proof of Proposition  \ref{prop:54} and the analogues phenomenon occurs in one dimensional $\beta$-expansions (cf.~\cite{Glendinning_Sidorov_2001, Kong_Li_Dekking_2010}) we prove Theorem \ref{thm:3} (ii).
\begin{proposition}\label{prop:betac-dim}
$\U_{\beta_c}$ is uncountable and $\dim_H\U_{\beta_c}=0$.
\end{proposition}
\begin{proof}
Note that $\beta_n\nearrow \beta_c$ as $n\ra\f$. 
By Lemma \ref{prop:betan} and the proof of Proposition \ref{prop:54} it follows that $\US_{\beta_c}$ contains all of the following sequences
\[
\su_1^{j_1}\su_2^{j_2}\cdots \su_{k}^{j_k}\cdots,
\]
where $j_k\in\N$. Observe that for each $k\ge 1$ the block $\su_k$ ends with digit $C$ and $\su_{k+1}=\su_k^B\Phi_A(\su_k^B)$. Then $\su_k$  can not be written as concatenation of two or more blocks of the form $\su_\ell$ with $\ell<k$. This implies that $\U_{\beta_c}$ is uncountable. 

Furthermore, by the recursive relation described in the proof of Proposition  \ref{prop:54} and using Propositions \ref{prop:betac-a} and \ref{prop:betac-b}  it follows that any sequence in $\US_{\beta_c}\setminus\US_{\beta_2}$ is of the form
\begin{align*}
&\a \;\Big(\su_{i_1}^B\Phi_A(\su_{i_1}^B)\Big)^{j_1}\left(\su_{i_1}^B\Phi_A(\su_{i_2'}^B)\right)^{j_1'}\cdots \Big(\su_{i_k}^B\Phi_A(\su_{i_k}^B)\Big)^{j_k}\left(\su_{i_k}^B\Phi_A(\su_{i_{k+1}'}^B)\right)^{j_k'}\cdots;\\
&\b \;\Big(\sv_{i_1}^C\Phi_B(\sv_{i_1}^C)\Big)^{j_1}\left(\sv_{i_1}^C\Phi_B(\sv_{i_2'}^C)\right)^{j_1'}\cdots \Big(\sv_{i_k}^C\Phi_B(\sv_{i_k}^C)\Big)^{j_k}\left(\sv_{i_k}^C\Phi_B(\sv_{i_{k+1}'}^C)\right)^{j_k'}\cdots;\\
&\c \;\Big(\sw_{i_1}^A\Phi_C(\sw_{i_1}^A)\Big)^{j_1}\left(\sw_{i_1}^A\Phi_C(\sw_{i_2'}^A)\right)^{j_1'}\cdots \Big(\sw_{i_k}^A\Phi_C(\sw_{i_k}^A)\Big)^{j_k}\left(\sw_{i_k}^A\Phi_C(\sw_{i_{k+1}'}^A)\right)^{j_k'}\cdots,
\end{align*}
where 
\[
\a, \b,\c\in\bigcup_{n=0}^\f\A^n; \quad j_k\in\N\cup\set{0, \f}, ~j_k'\in\set{0, 1}\]
{and}
\[1\le i_1<i_2'\le i_2<\cdots <i_k'\le i_k<i_{k+1}'\le \cdots.
\]
Observe that the length of $\su_k, \sv_k$ and $\sw_k$ is $2^{k-1}3$ which grows exponentially fast. This implies that $\dim_H\U_{\beta_c}=0$.
\end{proof}
We remark that for a    detailed proof of $\dim_H\U_{\beta_c}=0$ we refer to  \cite[Theorem 2.9]{Allaart-2017}  by an easy adaption.

\begin{proof}[Proof of Theorem \ref{thm:3}]
By Propositions \ref{prop:transcendental},  \ref{prop:54} and   \ref{prop:betac-dim} it suffices to prove that  for any $\beta>\beta_c$ we have $\dim_H\U_\beta>0$. We do this now
by first constructing a sequence of bases $(\hat\beta_n)$ such that  $\hat\beta_n$ strictly decreases to $\beta_c$ as $n\ra\f$, and then showing  that for any $\beta>\hat\beta_n$ the set $\US_\beta$ contains a subshift of finite type   with positive topological entropy which implies $\U_\beta>0$. 

For $n\ge 1$ let $\hat\beta_n\in(1,2)$ such that 
$
\de(\hat\beta_n)=\t_n^+(\re(\t_n))^\f.
$
By  Lemmas \ref{lem:delta-beta} and   \ref{l50} one can check that $\hat\beta_n$ is well-defined. Observe that $\de(\beta_c)$ begins with $\t_n^+\re(\t_n)\re(\t_n^+)$ for any $n\ge 1$. Then  
\[
\de(\hat\beta_1)\succ\de(\hat\beta_2)\succ\cdots\quad\textrm{and}\quad \de(\hat\beta_n)\searrow(\la_i)=\de(\beta_c)\quad\textrm{as }n\ra\f.
\]
So by Lemma \ref{lem:delta-beta} we have 
\begin{equation}\label{eq-hatbetan}
\hat\beta_1>\hat\beta_2>\cdots\quad\textrm{and}\quad \hat\beta_n\searrow \beta_c\quad \textrm{as }n\ra\f.
\end{equation}

Now take $\beta\in(\hat\beta_n, \hat\beta_{n-1}]$ with $n\ge 2$. Let $(X_n, \si)$ be the subshift of finite type represented by the labeled graph $\mathcal G$ as in Figure \ref{figure:15}. Similar to the proof of Lemma \ref{prop:betan} one can prove by using  Lemmas \ref{l50}, \ref{lem:51}, \ref{lem:52} and Corollary \ref{cor:54}  that any sequence in the subshift of finite type satisfies the conditions in Proposition \ref{prop:21}. In other words, $X_n\subseteq\US_\beta$.
\begin{figure}[h!]
  \centering
 \begin{tikzpicture}[->,>=stealth',shorten >=1pt,auto,node distance=4cm,
                    semithick]

  \tikzstyle{every state}=[minimum size=0pt,fill=black,draw=none,text=black]

  \node[state] (A)                    { };
  \node[state]         (B) [ right of=A] { };

  \path[->,every loop/.style={min distance=0mm, looseness=80}]
   (A) edge [loop left,->]  node {$\Phi_A(\su_n)$} (A)
            edge  [bend left]   node {$\Phi_A(\su_n^B)$} (B)

        (B) edge [loop right] node {$\su_n$} (B)
            edge  [bend left]            node {$\su_n^B$} (A);
\end{tikzpicture}
  \caption{The   labeled graph $\mathcal G$ presenting $(X_n, \si)$.}\label{figure:15}
\end{figure}
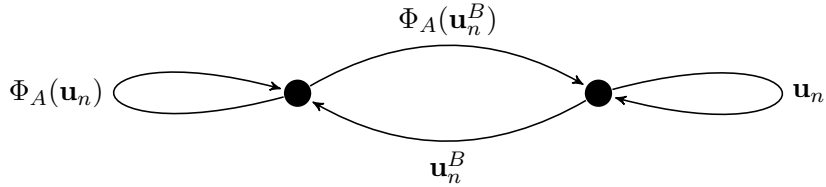
 Note that the   the blocks $\su_n^B, \Phi_A(\su_n), \Phi_A(\su_n^B)$ and $\su_n$ have the same length   $3\cdot 2^{n-1}$. Then   the topological entropy of $X_n$ is given by
(cf.~\cite{Lind_Marcus_1995})
\begin{equation}\label{eq:top-entropy}
h(X_n)=\frac{\log 2}{\log (2^{n-1}3)}\ge \frac{1}{n+1}.
\end{equation}
One can verify that the projection $\set{((d_i))_\beta: (d_i)\in X_n}$ is a graph-directed set satisfying the open set condition (cf.~\cite{Mauldin_Williams_1988}). Then by (\ref{eq:top-entropy}) it follows that 
\[
\dim_H\U_\beta= \frac{h(X_n)}{\log\beta}\ge \frac{1}{(n+1)  \log \beta}.
\]
This together with (\ref{eq-hatbetan}) implies that $\dim_H\us_\beta>0$ for any $\beta>\beta_c$. 
\end{proof}

\section{Open questions}\label{sec:final remark}
Note that $\U_\beta\subseteq \u_\beta$ for any $\beta\in(1,2)$. Furthermore, by Proposition \ref{lem:22} the two sets $\U_\beta$ and $\u_\beta$ coincide if $\beta\in(1,3/2]$ or $\beta$ is a multinacci number. Then it is natural to investigate the difference between $\u_\beta$ and $\U_\beta$ for other $\beta$s.  
 
\begin{question}
Can we describe the set of $\beta\in(3/2,2)$ for which $\u_\beta=\U_\beta$?
Is it true that $\dim_H\u_\beta=\dim_H\U_\beta$ for all $\beta\in(1,2)$? \end{question}

Observe that the set-valued map $\beta\mapsto\US_\beta$ is constant on each interval $(\beta_n, \beta_{n+1}]\subset(1,\beta_c)$. Motivated by the work of de Vries and Komornik \cite{DeVries_Komornik_2008} we ask the following question.
\begin{question}
Is it true that the set-valued map $\beta\mapsto \US_\beta$ is locally constant for Lebesgue almost every $\beta\in(1,2)$? If so, can we describe the bifurcation set \[\mathcal V=\set{\beta\in(1,2): \US_{\beta'}\ne\US_\beta\quad\textrm{for any }\beta'>\beta}?\]
\end{question} 

By Theorem \ref{thm:3} we know that $\dim_H\U_\beta=0$ for all $\beta\le \beta_c$. For $\beta>\beta_c$ the authors in \cite{Bro-Mon-Sid-04} calculated  the dimension of $\U_\beta$ only for $\beta$ being a multinacci number. In view of the work by Komornik et al.~\cite{Komornik-Kong-Li-17} and Alcaraz Barrera et al.~\cite{AlcarazBarrera-Baker-Kong-2016} we ask the following analogous question.
\begin{question}
Can we give a uniform formula for the Hausdorff dimension of $\U_\beta$ for $\beta>\beta_c$? Is it true that the entropy function $\beta\mapsto h(\US_\beta)$ is a Devil's staircase, where $h(\US_\beta)$ denotes the topological entropy of $\US_\beta$? If so, can we describe the bifurcation set 
\[\mathcal B=\set{\beta\in(1,2): h(\US_{\beta'})\ne h(\US_\beta)\quad\textrm{for any }\beta'>\beta}?\]
\end{question} 

\section*{Acknowledgements}
The first author was supported by NSFC No.~11471124. The second author was supported by NSFC No.~11671147, 11571144 and Science and Technology Commission of Shanghai Municipality (STCSM)  No.~18dz2271000.


\end{document}